\documentclass[reqno]{amsart}
\usepackage[utf8]{inputenc}

\usepackage[utf8]{inputenc} 
\usepackage{amssymb}
\usepackage{mathrsfs}
\usepackage{graphicx}
\usepackage{latexsym}
\usepackage{stmaryrd}
\usepackage{enumitem}
\usepackage[bookmarks,hyperfootnotes=false]{hyperref}
\usepackage{multirow}
\usepackage{xcolor}
\usepackage{caption}
\colorlet{darkishRed}{red!80!black}
\colorlet{darkishBlue}{blue!60!black}
\colorlet{darkishGreen}{green!60!black}
\hypersetup{
    draft = false,
    bookmarksopen=true,
    colorlinks,
    linkcolor={red!60!black},
    citecolor={green!60!black},
    urlcolor={blue!60!black}
}
\usepackage[nameinlink, capitalise, noabbrev]{cleveref}
\crefformat{enumi}{#2#1#3}
\crefformat{equation}{#2(#1)#3}
\usepackage[msc-links]{amsrefs} 
\usepackage{doi}

\renewcommand{\PrintDOI}[1]{\doi{#1}}

\let\setminus=\smallsetminus
\usepackage{tikz}
\usepackage{tikz-cd}
\usetikzlibrary{calc,through,intersections,arrows, trees, positioning, decorations.pathmorphing, cd}
\usepackage{relsize}
\usepackage{comment}
\usepackage{svg}
\usepackage{mathtools}
\usepackage{nccmath}
\usepackage{pifont}
\usepackage{comment}

\usepackage[skip=6pt]{subcaption} 
\usepackage{pgfplots}
\pgfplotsset{compat=1.15}
\usepackage{mathrsfs}
\usetikzlibrary{arrows}

\usepackage{geometry}
\let\setminus=\smallsetminus
\renewcommand{\leq}{\leqslant}
\renewcommand{\geq}{\geqslant}

\usepackage{aligned-overset}

\let\rho=\varrho
\let\phi=\varphi

\newcommand{\mathintitle}[1]{\textbf{$\boldsymbol{#1}$}}

\newcommand{\shiftingfct}[2]{f\mkern-6mu\downarrow_{#1}^{#2}}

\newcommand{ \R } { \mathbb{R} }

\newcommand{ \N } { \mathbb{N} }

\newcommand{\abs}[1]{\lvert#1\rvert}

\makeatletter

\def\calCommandfactory#1{%
   \expandafter\def\csname c#1\endcsname{\mathcal{#1}}}
\def\frakCommandfactory#1{%
   \expandafter\def\csname frak#1\endcsname{\mathfrak{#1}}}

\newcounter{ctr}
\loop
  \stepcounter{ctr}
  \edef\X{\@Alph\c@ctr}
  \expandafter\calCommandfactory\X
  \expandafter\frakCommandfactory\X
  \edef\Y{\@alph\c@ctr}
  \expandafter\frakCommandfactory\Y
\ifnum\thectr<26
\repeat

\setenumerate{label={\normalfont (\roman*)}}

\usepackage[presets={vec-cev,abc,ABC,cAcBcC}]{letterswitharrows}

\newtheorem{theorem}{Theorem}[section] 

\newtheorem{mainresult}{Theorem}

\newtheorem{LEM}[theorem]{Lemma}
\newtheorem{THM}[theorem]{Theorem}
\newtheorem{COR}[theorem]{Corollary}

\newtheorem{PROP}[theorem]{Proposition}

\theoremstyle{definition}

\crefname{claim}{Claim}{Claims}
\AtEndEnvironment{proof}{\setcounter{claim}{0}}

\usepackage{etoolbox}

\def\namedlabel#1#2{\begingroup
   \def\@currentlabel{#2}%
   \label{#1}\endgroup
}

\newlist{thmlist}{enumerate}{1}
\setlist[thmlist]{label=(\roman{thmlisti}), ref=\thethm.(\roman{thmlisti}),noitemsep}

\newtheorem{innercustomthm}{Theorem}

\newenvironment{customthm}[1]
  {\renewcommand\theinnercustomthm{#1}\innercustomthm}
  {\endinnercustomthm}

\theoremstyle{definition}

\newtheorem{EX}[theorem]{Example}
\newtheorem{CONSTR}[theorem]{Construction}

\theoremstyle{remark}

\def\bigskip{\vspace{14pt}}

\title[Refining trees of tangles in abstract separation systems: inessential parts]{Refining trees of tangles in abstract separation systems:\\ inessential parts}

\author{Sandra Albrechtsen}
\address{University of Hamburg, Department of Mathematics, Bundesstraße 55 (Geomatikum), 20146 Hamburg, Germany}
\email{sandra.albrechtsen@uni-hamburg.de}


\begin{document}

\maketitle

\begin{abstract}
\noindent Robertson and Seymour proved two fundamental theorems about tangles in graphs: the \emph{tree-of-tangles theorem}, which says that every graph has a tree-decomposition such that distinguishable tangles live in different nodes of the tree, and the \emph{tangle-tree duality theorem}, which says that graphs without a $k$-tangle have a tree-decomposition that witnesses the non-existence of such tangles, in that $k$-tangles would have to live in a node but no node is large enough to accommodate one.

Erde combined these two fundamental theorems into one, by constructing a single tree-decomposition such that every node either accommodates a single $k$-tangle or is too small to accommodate one. Such a tree-decomposition thus shows at a glance how many $k$-tangles a graph has and where they are.

The two fundamental theorems have since been extended to abstract separation systems, which support tangles in more general discrete structures.
In this paper we extend Erde's unified theorem to such general systems.
\end{abstract}

\section{Introduction}

\emph{Tangles} were introduced by Robertson and Seymour as a way to indirectly capture highly cohesive but fuzzy structures inside a graph \cite{GMX}. Formally, a tangle of a graph $G$ is an orientation of all its separations up to some order. The idea is that every highly cohesive substructure of $G$ will lie mostly on one side of any such low-order separation, and thus `orients' it towards that side. A tangle, very broadly, is an orientation of all the low-order separations that arises in this way.

Since its first introduction, the notion of a tangle and its framework of graph separations have been generalized to so-called \emph{abstract separation systems}. 
Although these are significantly more general than the separation systems of graphs, the two fundamental theorems about graph tangles~-- the tree-of-tangles theorem and the tangle-tree duality theorem from Robertson and Seymour \cite{GMX}~-- are still valid in this setting.  

In what follows we assume familiarity with graph tangles as described in \cite{DiestelBook16noEE}*{Ch.\ 12.5}.
\bigskip

For a given graph $G$ and an integer $k > 0$, the \emph{tree-of-tangles theorem} asserts the existence of a tree-decomposition $(T,\cV)$ of $G$ that distinguishes all its $k$-tangles, in that they live in different parts. A part~$V_t \in \cV$ is called \emph{essential} if there is a tangle living in it, and otherwise \emph{inessential}.
Choosing the tree~$T$ of the decomposition to be minimal so that it is still able to distinguish all the tangles ensures that all its parts are essential.

However, in some applications it is important to have a canonical tree-decomposition, one which can be defined purely in terms of invariants of the underlying graph. For example, an algorithm which constructs the tree-decomposition canonically will always produce the same output for a given input graph and set of tangles, regardless of how these are given to the algorithm as input.

A canonical tree-decomposition that distinguishes the $k$-tangles in a graph may have to have some inessential parts, since the canonicity requirement does not allow us to merge inessential parts with essential ones in order to make all the parts essential. These inessential parts can contain a large portion of the graph, and then the tree-decomposition tells us nothing about the structure of that portion. In particular, it does not tell us \emph{why} there are no tangles living in these parts.

However, if there is no tangle (of some given order) in $G$ at all, then the tangle-tree duality theorem does tell us something about its structure: it guarantees the existence of a tree-decomposition in which each part is too `small' to be home to a tangle. Since every tangle has to live in some part, this tree-decomposition then witnesses that there are no tangles in $G$ at all.

As inessential parts of a tangle-distinguishing tree-decomposition are not home to any tangles, it seems natural to ask whether it is possible to locally find tree-decompositions of the inessential parts which refine the original canonical tree-decomposition and witness that these parts are inessential, in the same way as a tree-decomposition from the tangle-tree duality theorem witnesses that there is no tangle at all.

\begin{figure}[h]
  \centering
  \begin{subfigure}[b]{0.35\linewidth}
  \definecolor{grey}{rgb}{0.6,0.6,0.6}
\scalebox{0.63}{%
\begin{tikzpicture}
\draw [line width=1.1pt,color=grey] (11,6) circle (3cm);
\draw [line width=1.1pt,color=grey] (8.061541557078904,8.60471544394863) circle (1.5cm);
\draw [line width=1.1pt,color=grey] (13.977134821732324,8.503421038731961) circle (1.5cm);
\draw [line width=1.1pt,color=grey] (9.560698754285593,2.34472120155852) circle (1.5cm);
\draw [line width=1.1pt] (10.025377349769181,3.53739626329973)-- (10.92661120692412,4.215787356752289);
\draw [line width=1.1pt] (9.056414991584486,7.799341711253636)-- (9.953904750352931,7.542725382126214);
\draw [line width=1.1pt] (12.983708901583505,7.696262478611045)-- (12.198877317346657,7.117205530104936);
\draw [line width=1.1pt] (12.198877317346657,7.117205530104936)-- (12.009034051586545,6.210419266702657);
\draw [line width=1.1pt] (12.009034051586545,6.210419266702657)-- (12.842637483305642,5.829248741011914);
\draw [line width=1.1pt] (12.009034051586545,6.210419266702657)-- (11.222652302402308,5.949742344553106);
\draw [line width=1.1pt] (11.222652302402308,5.949742344553106)-- (10.24994584583112,5.949742344553106);
\draw [line width=1.1pt] (10.24994584583112,5.949742344553106)-- (10.630570111445932,5.103910643186854);
\draw [line width=1.1pt] (11.222652302402308,5.949742344553106)-- (10.630570111445932,5.103910643186854);
\draw [line width=1.1pt] (10.630570111445932,5.103910643186854)-- (10.92661120692412,4.215787356752289);
\draw [line width=1.1pt] (10.24994584583112,5.949742344553106)-- (9.338231569899573,5.735533094263147);
\draw [line width=1.1pt] (10.24994584583112,5.949742344553106)-- (10.292237430899432,6.725088070805504);
\draw [line width=1.1pt] (10.292237430899432,6.725088070805504)-- (9.728349629988598,6.725088070805504);
\draw [line width=1.1pt] (9.728349629988598,6.725088070805504)-- (9.953904750352931,7.542725382126214);
\draw [line width=1.1pt] (9.953904750352931,7.542725382126214)-- (10.581762246656849,7.531077635120268);
\draw [line width=1.1pt] (10.581762246656849,7.531077635120268)-- (10.292237430899432,6.725088070805504);
\draw [line width=1.1pt] (7.066668122573324,9.410089176643625)-- (6.8666307277715495,8.145816642464021);
\draw [line width=1.1pt] (6.8666307277715495,8.145816642464021)-- (7.86150416227713,7.340442909769026);
\draw [line width=1.1pt] (7.86150416227713,7.340442909769026)-- (9.056414991584486,7.799341711253636);
\draw [line width=1.1pt] (9.056414991584486,7.799341711253636)-- (9.256452386386258,9.06361424543324);
\draw [line width=1.1pt] (9.256452386386258,9.06361424543324)-- (8.261578951880677,9.868987978128233);
\draw [line width=1.1pt] (8.261578951880677,9.868987978128233)-- (7.066668122573324,9.410089176643625);
\draw [line width=1.1pt] (7.066668122573324,9.410089176643625)-- (9.256452386386258,9.06361424543324);
\draw [line width=1.1pt] (9.256452386386258,9.06361424543324)-- (7.86150416227713,7.340442909769026);
\draw [line width=1.1pt] (7.86150416227713,7.340442909769026)-- (7.066668122573324,9.410089176643625);
\draw [line width=1.1pt] (7.066668122573324,9.410089176643625)-- (9.056414991584486,7.799341711253636);
\draw [line width=1.1pt] (9.056414991584486,7.799341711253636)-- (8.261578951880677,9.868987978128233);
\draw [line width=1.1pt] (8.261578951880677,9.868987978128233)-- (6.8666307277715495,8.145816642464021);
\draw [line width=1.1pt] (6.8666307277715495,8.145816642464021)-- (9.056414991584486,7.799341711253636);
\draw [line width=1.1pt] (7.86150416227713,7.340442909769026)-- (8.261578951880677,9.868987978128233);
\draw [line width=1.1pt] (6.8666307277715495,8.145816642464021)-- (9.256452386386258,9.06361424543324);
\draw [line width=1.1pt] (13.774827963859952,9.767332402419228)-- (14.179441679604697,7.239509675044694);
\draw [line width=1.1pt] (14.179441679604697,7.239509675044694)-- (14.970560741881144,9.310579598852877);
\draw [line width=1.1pt] (14.970560741881144,9.310579598852877)-- (12.781402043711132,8.960173842298312);
\draw [line width=1.1pt] (12.781402043711132,8.960173842298312)-- (15.172867599753516,8.04666823516561);
\draw [line width=1.1pt] (15.172867599753516,8.04666823516561)-- (12.983708901583505,7.696262478611045);
\draw [line width=1.1pt] (12.983708901583505,7.696262478611045)-- (13.774827963859952,9.767332402419228);
\draw [line width=1.1pt] (13.774827963859952,9.767332402419228)-- (14.970560741881144,9.310579598852877);
\draw [line width=1.1pt] (14.970560741881144,9.310579598852877)-- (15.172867599753516,8.04666823516561);
\draw [line width=1.1pt] (15.172867599753516,8.04666823516561)-- (14.179441679604697,7.239509675044694);
\draw [line width=1.1pt] (14.179441679604697,7.239509675044694)-- (12.983708901583505,7.696262478611045);
\draw [line width=1.1pt] (12.983708901583505,7.696262478611045)-- (12.781402043711132,8.960173842298312);
\draw [line width=1.1pt] (12.781402043711132,8.960173842298312)-- (13.774827963859952,9.767332402419228);
\draw [line width=1.1pt] (13.774827963859952,9.767332402419228)-- (15.172867599753516,8.04666823516561);
\draw [line width=1.1pt] (14.179441679604697,7.239509675044694)-- (12.781402043711132,8.960173842298312);
\draw [line width=1.1pt] (12.983708901583505,7.696262478611045)-- (14.970560741881144,9.310579598852877);
\draw [line width=1.1pt] (8.760151150099325,3.343482200712786)-- (8.295472554615737,2.150807138971576);
\draw [line width=1.1pt] (8.295472554615737,2.150807138971576)-- (9.096020158802004,1.1520461398173103);
\draw [line width=1.1pt] (9.096020158802004,1.1520461398173103)-- (10.36124635847186,1.345960202404254);
\draw [line width=1.1pt] (10.36124635847186,1.345960202404254)-- (10.825924953955449,2.538635264145464);
\draw [line width=1.1pt] (10.825924953955449,2.538635264145464)-- (10.025377349769181,3.53739626329973);
\draw [line width=1.1pt] (10.025377349769181,3.53739626329973)-- (8.760151150099325,3.343482200712786);
\draw [line width=1.1pt] (8.760151150099325,3.343482200712786)-- (10.36124635847186,1.345960202404254);
\draw [line width=1.1pt] (10.36124635847186,1.345960202404254)-- (8.295472554615737,2.150807138971576);
\draw [line width=1.1pt] (8.295472554615737,2.150807138971576)-- (10.025377349769181,3.53739626329973);
\draw [line width=1.1pt] (10.025377349769181,3.53739626329973)-- (10.36124635847186,1.345960202404254);
\draw [line width=1.1pt] (9.096020158802004,1.1520461398173103)-- (10.825924953955449,2.538635264145464);
\draw [line width=1.1pt] (10.825924953955449,2.538635264145464)-- (8.760151150099325,3.343482200712786);
\draw [line width=1.1pt] (8.760151150099325,3.343482200712786)-- (9.096020158802004,1.1520461398173103);
\draw [line width=1.1pt] (9.096020158802004,1.1520461398173103)-- (10.025377349769181,3.53739626329973);
\draw [line width=1.1pt] (8.295472554615737,2.150807138971576)-- (10.825924953955449,2.538635264145464);

\begin{scriptsize}
\draw [fill=black] (10.025377349769181,3.53739626329973) circle (2.3pt);
\draw [fill=black] (12.983708901583505,7.696262478611045) circle (2.3pt);
\draw [fill=black] (9.056414991584486,7.799341711253636) circle (2.3pt);
\draw [fill=black] (9.256452386386258,9.06361424543324) circle (2.3pt);
\draw [fill=black] (8.261578951880677,9.868987978128233) circle (2.3pt);
\draw [fill=black] (7.066668122573324,9.410089176643625) circle (2.3pt);
\draw [fill=black] (6.8666307277715495,8.145816642464021) circle (2.3pt);
\draw [fill=black] (7.86150416227713,7.340442909769026) circle (2.3pt);
\draw [fill=black] (14.179441679604697,7.239509675044694) circle (2.3pt);
\draw [fill=black] (15.172867599753516,8.04666823516561) circle (2.3pt);
\draw [fill=black] (14.970560741881144,9.310579598852877) circle (2.3pt);
\draw [fill=black] (13.774827963859952,9.767332402419228) circle (2.3pt);
\draw [fill=black] (12.781402043711132,8.960173842298312) circle (2.3pt);
\draw [fill=black] (8.760151150099325,3.343482200712786) circle (2.3pt);
\draw [fill=black] (8.295472554615737,2.150807138971576) circle (2.3pt);
\draw [fill=black] (9.096020158802004,1.1520461398173103) circle (2.3pt);
\draw [fill=black] (10.36124635847186,1.345960202404254) circle (2.3pt);
\draw [fill=black] (10.825924953955449,2.538635264145464) circle (2.3pt);
\draw [fill=black] (10.92661120692412,4.215787356752289) circle (2.3pt);
\draw [fill=black] (9.953904750352931,7.542725382126214) circle (2.3pt);
\draw [fill=black] (12.198877317346657,7.117205530104936) circle (2.3pt);
\draw [fill=black] (12.009034051586545,6.210419266702657) circle (2.3pt);
\draw [fill=black] (12.842637483305642,5.829248741011914) circle (2.3pt);
\draw [fill=black] (11.222652302402308,5.949742344553106) circle (2.3pt);
\draw [fill=black] (10.24994584583112,5.949742344553106) circle (2.3pt);
\draw [fill=black] (10.630570111445932,5.103910643186854) circle (2.3pt);
\draw [fill=black] (9.338231569899573,5.735533094263147) circle (2.3pt);
\draw [fill=black] (10.292237430899432,6.725088070805504) circle (2.3pt);
\draw [fill=black] (9.728349629988598,6.725088070805504) circle (2.3pt);
\draw [fill=black] (10.581762246656849,7.531077635120268) circle (2.3pt);
\end{scriptsize}
\end{tikzpicture}
}%

\subcaption{}
    \label{fig:RefiningAToT1}
  \end{subfigure}
\hspace{10mm}
  \begin{subfigure}[b]{0.35\linewidth}
    \include{RefiningToTsIntro2}
    \label{fig:RefiningAToT2}
  \end{subfigure}
  \centering
  \caption{A tree-decomposition of $G$ that distinguishes all its $k$-tangles, and a refinement of the inessential part.}
  \label{fig:RefiningAToT}
\end{figure}
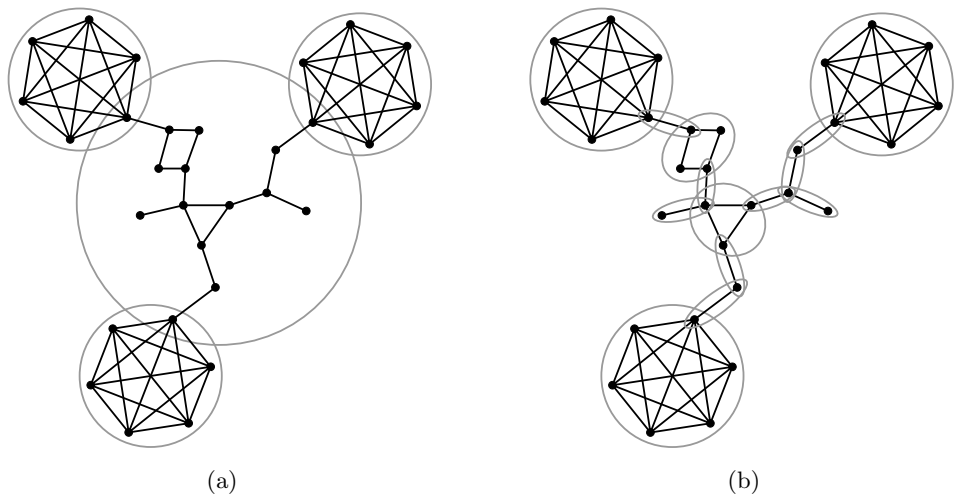

For example, think of a tree-like graph to which three large disjoint cliques have been glued in a vertex. If $k$ is chosen big enough, there will be precisely three $k$-tangles in the resulting graph $G$, those induced by the cliques. A canonical tree-decomposition that distinguishes these $k$-tangles could look like a star with three leaves, where each leaf contains one of the three cliques while the center of the star contains the rest of $G$ (\cref{fig:RefiningAToT1}).

However, in this example it is easy to see that the middle part can be split into smaller parts that are each too small to be home to a $k$-tangle. This gives rise to one overall tree-decomposition of $G$ in which the three big parts are each home to a $k$-tangle and all other parts are too small to accommodate one (\cref{fig:RefiningAToT2}).

Erde \cite{JoshRefining} showed that for every graph $G$ there is a canonical tree-decomposition of $G$ that efficiently distinguishes a given set of $k$-tangles in $G$ and which extends to a tree-decomposition all whose inessential parts are too small to be home to a tangle: 

\newtheorem*{ERDE*}{\textbf{\emph{Erde's Refinement Theorem}}}\namedlabel{thm:ErdesRefThm}{Erde's Refinement Theorem}
\begin{ERDE*}{\cite{JoshRefining}*{Corollary~3.2}}
	\emph{Let $G$ be a graph and $k \geq 3$ and let $\cF \subseteq 2^{\vS_k(G)}$ be a friendly canonical set of stars. Then there exist nested sets $\tilde{N} \subseteq N \subseteq S_k(G)$ such that:
	\begin{itemize}
		\item $\tilde{N}$ is fixed under all automorphisms of $G$ and distinguishes all the $\cF$-tangles of $S_k(G)$ efficiently;
		\item every node of $N$ is either a star in $\cF$ or home to a tangle.
	\end{itemize}}
\end{ERDE*}

\noindent See \cref{sec:ASS} for definitions.
\medskip 

In fact, Erde \cite{JoshRefining} showed more than this: it is possible to refine \emph{every} such tree set $\tilde{N}$ by some $N$ as above, as long as every separation in $\tilde{N}$ distinguishes some pair of $\cF$-tangles efficiently. He also gave an example \cite{JoshRefining} which shows that this additional assumption about $\tilde{N}$ is necessary to get the sharper result even if $\cF$ consists only of the stars whose exclusion defines $k$-tangles as such.

Note also that we cannot strengthen the theorem so that $N$, too, is canonical (see \cite{JoshRefining} for an example).
\bigskip

Since Robertson and Seymour introduced tangles for graphs, an effort was made to generalize the concept of tangles and separations so as to make them applicable in as many scenarios as possible. 
The original definition of a tangle was already made in such a way that it could easily be generalized to set separations~\cite{ASS}. This led to the notion of a tangle for `connectivity systems', such as those arising from matroids. 

Later, Diestel \cite{ASS} introduced \emph{abstract separation systems}, which define `separations' axiomatically, without reference to any underlying graph or set. Instead, an \emph{abstract separation system} is defined as a poset with an order-reversing involution. Such a separation system $\vS$ is called \emph{submodular} if every two of its elements, $\vr$ and $\vs$, say, have an infimum $\vr \wedge \vs$ and a supremum $\vr \vee \vs$ in some fixed larger `universe' $\vU \supseteq \vS$ of separations, and at least one of $\vr \wedge \vs$ and $\vr \vee \vs$ always lies in $\vS$. This holds for separation systems of the form $\vS_k(G) := \{\vs : \abs{s} < k\}$ of a graph $G$, where $\vr \wedge \vs$ and $\vr \vee \vs$ are opposite corner separations of the separations $\vr$ and $\vs$ \cite{DiestelBook16noEE}. 

Submodularity makes it possible to generalize the fundamental theorems of the tangle theory of graphs, such as the tree-of-tangles theorem and the tangle-tree duality theorem, to abstract separation systems \cite{AbstractTangles}*{Theorem 6 and 4}. In this more general setting, tangles are distinguished by a \emph{tree of tangles} which generalizes the concept of tree-decompositions.
\bigskip

In this paper we show that Erde's theorem generalizes too: we can still find a canonical tree of tangles that distinguishes a given set of tangles, and a refinement such that each node of the refined tree of tangles is either home to a tangle or too small for that:

\begin{mainresult}\label{cor:RefiningACanonicalToT}
    Let $\vS$ be a submodular separation system, and let $\cF$ be a friendly canonical set of stars in $\vS$.
	Then there are nested sets $\tilde{N} \subseteq N \subseteq S$ such that:
	\begin{itemize}
		\item $\tilde{N}$ is fixed under all automorphisms of $\vS$ and distinguishes all the $\cF$-tangles of $S$;
		\item every node of $N$ is either a star in $\cF$ or home to a tangle.
	\end{itemize}
\end{mainresult}

\cref{cor:RefiningACanonicalToT} is a direct generalization of Erde's Refinement Theorem to abstract separation systems. Like Erde's result, \cref{cor:RefiningACanonicalToT} asserts the existence of a canonical tree of tangles that can be refined so that all inessential nodes become stars in $\cF$. 

As remarked after the statement of Erde's theorem, it is in fact possible in separation systems of the form $S_k$, to refine \emph{every} tree of tangles in which every separation distinguishes a pair of $\cF$-tangles efficiently. It is therefore natural to ask whether we can get some analogue of this for abstract separation systems as well.

But we run into a subtle difficulty here: as pointed out earlier, we cannot hope to refine every tree set $\tilde{N}$ without imposing further conditions on the separations inside $\tilde{N}$. But the notion of efficiency, to which this condition refers in the case of Erde's theorem, is not defined in general as it appeals to the `order' of separations.
Instead, we will require the separations to be `good': a property that is a structural generalization of efficiently distinguishing a pair of tangles to abstract separation systems.

As long as we focus on good separations, we can show the following variant of \cref{cor:RefiningACanonicalToT}:

\begin{mainresult}\label{thm:RefiningGoodToTs}
	Let $\vS$ be a submodular separation system, and let $\cF$ be a friendly set of stars in $\vS$. Further, let $\tilde{N} \subseteq S$ be a nested set of good separations that distinguishes all $\cF$-tangles of $S$.  
	Then there exists a nested set $N \subseteq S$ with $\tilde{N} \subseteq N$ such that every node of $N$ is either a star in $\cF$ or home to an $\cF$-tangle.
\end{mainresult}

We remark that in the case where the separation system comes from a graph $G$, every separation which efficiently distinguishes a pair of tangles in $G$ is good.
We will show in \cref{sec:ANearlyCanonicalToTWithGoodSeps} that there always exist tangle-distinguishing tree sets $\tilde{N}$ that contain only good separations.

\section{Abstract separation systems and tangles}\label{sec:ASS}

In this section we give a short overview of the important objects, definitions and theorems that we are going to use later. For a more detailed introduction to abstract separation systems and their tangles we refer the reader to \cite{AbstractTangles} and \cite{ASS}.
The definitions are taken from \cites{TangleTreeAbstract, AbstractTangles, ASS, ProfilesNew, TreeSets}. We use the basic graph-theoretic notions from \cite{DiestelBook16noEE}.

\subsection{Separation systems}

A \emph{separation system} $\vS = (\vS, \leq, ^*)$ is a partially ordered set $(\vS, \leq)$ with an order-reversing involution $^*$, i.e.\ if $\vs \leq \vr \in \vS$, then $(\vs)^* \geq (\vr)^*$. We denote $(\vs)^*$ by $\sv$ and $(\sv)^*$ by $\vs$. The elements $\vs, \sv \in \vS$ are called \emph{oriented separations}. An \emph{unoriented separation} is a set $s := \{\vs, \sv\}$ for some~$\vs \in \vS$ and~$\vs$,~$\sv$ are the \emph{orientations} of $s$. Note that there are no default orientations: once we denoted one orientation by $\vs$ the other one will be $\sv$, and vice versa. The set of all sets $\{\vs, \sv\} \subseteq \vS$ is denoted by $S$. We will use terms defined for unoriented separations also for oriented ones and vice versa if that is possible without causing ambiguities. Moreover, if the context is clear, we will simply refer to both oriented and unoriented separations as `separations'.

Two unoriented separations $r, s \in S$ are \emph{nested} if they have orientations that can be compared, otherwise they \emph{cross}. 
If $r, s \in S$ are nested, then both $r$ and $\vr$ are \emph{nested} with both of $s, \vs$. Analogously, if $r, s \in S$ cross, then both $r, \vr$ \emph{cross} both of $s, \vs$.
A set of separations is \emph{nested} if all its elements are pairwise nested.
Two sets $R, R'$ of separations are \emph{nested} if every element of $R$ is nested with every element of $R'$.

If a separation $s \in S$ satisfies $\vs = \sv$, we call $s$ \emph{degenerate}. 
A separation $\vs \in \vS$ is \emph{trivial} in $\vS$ if there exists a separation $\vr \in \vS$ such that $\vs < \vr$ and $\vs < \rv$. 
A set $\cF \subseteq 2^{\vS}$ of subsets of $\vS$ is \emph{standard (for $\vS$)} if~$\{\rv\} \in \cF$ for every trivial separation $\vr \in \vS$.
A separation $\vs \in \vS$ is \emph{small} and $\sv$ \emph{co-small} if $\vs \leq \sv$.

A set $\sigma \subseteq \vS$ of non-degenerate separations is called a \emph{star} if for any $\vr, \vs \in \sigma$ it holds that $\vr \leq \sv$.

We call a separation system $(\vU, \leq, ^*)$ a \emph{universe} of separations, and denote it with $\vU = (\vU, \leq, ^*, \wedge, \vee)$, if it is a lattice, i.e.\ if any two separations $\vr, \vs \in \vU$ have a \emph{supremum} $\vr \vee \vs$ and an \emph{infimum} $\vr \wedge \vs$ in $\vU$.

Given two separations $r, s \in U$ we call the four separations $\{(\vr \vee \vs), (\vr \vee \vs)^*\}$, $\{(\vr \wedge \vs), (\vr \wedge \vs)^*\}$, $\{(\rv \vee \vs), (\rv \vee \vs)^*\}$ and $\{(\rv \wedge \vs), (\rv \wedge \vs)^*\}$ their \emph{corner separations}. A simple but quite useful observation about corner separations is the following:

\begin{LEM}{\cite{ASS}*{Lemma 3.2}}\label{lem:Fischlemma}
	Let $r, s \in U$ be two crossing separations. Every separation
	$t \in U$ that is nested with both $r$ and $s$ is also nested with all four corner separations of $r$ and $s$.
\end{LEM}

A function $\vert\! \cdot\! \vert: \vU \rightarrow \R_{\geq 0}$ is called an \emph{order function} of $\vU$ if $\abs{\vs} = \abs{\sv}$ for all $s \in U$. If $\vU$ comes with an order function that is \emph{submodular}, i.e.\
\[
\abs{\vr \vee \vs} + \abs{\vr \wedge \vs} \leq \abs{\vr} + \abs{\vs} \text{ for all } \vr, \vs \in \vU,
\]
then we call $\vU$ a \emph{submodular} universe. $\abs{\vs}$ is then the \emph{order} of $\vs$ and $s$. For an integer $k > 0$, the induced sets $\vS_k := \{\vs \in \vU: \abs{\vs} < k\}$ are separation systems on their own. Though they need not to be universes, as in general $\vr \vee \vs$ and $\vr \wedge \vs$ for two separations $\vr, \vs \in \vS_k$ need not both lie in $\vS_k$. Note that we take the infimum and supremum here with respect to $\vU$. However, by the submodularity of $\vert\! \cdot\! \vert$, at least one of $\vr \vee \vs$ and $\vr \wedge \vs$ has to be contained in $\vS_k$.

This property motivates the following structural formulation of submodularity which can be applied to universes without making use of the external concept of an order function:

A separation system $\vS$ is \emph{inside} a universe $\vU$ if $\vS \subseteq \vU$ and the partial order and the involution on $\vS$ are the ones induced by $\vU$. 
Then, a separation system $\vS$ inside some universe $\vU$ is called \emph{submodular} if
\[
\text{for all } \vr, \vs \in \vS : \vr \vee \vs \in \vS \text{ or } \vr \wedge \vs \in \vS.
\]

\subsection{Profiles and tangles in abstract separation systems}\label{sec:profiles}

An \emph{orientation} of a set $S$ of unoriented separations is a set $O \subseteq \vS$ which contains every degenerate separation from $\vS$ and exactly one orientation~$\vs$ or~$\sv$ of every non-degenerate separation in $S$. A subset $O \subseteq \vS$ is \emph{consistent} if it does not contain both $\rv$ and $\vs$ whenever $\vr < \vs$ for distinct $r, s \in S$. 
If $\cO$ is a set of consistent orientations of $S$, then we call a star~$\sigma \subseteq \vS$ \emph{essential (for $\cO$)} if $\sigma \subseteq O$ for some $O \in \cO$. Otherwise $\sigma$ is called \emph{inessential (for $\cO$)}. 

A non-degenerate separation $s \in S$ \emph{distinguishes} two orientations $O_1$ and $O_2$ of $S$ if $O_1$ and $O_2$ orient~$s$ differently. A set of separations $N \subseteq S$ \emph{distinguishes} a set $\cO$ of orientations if any two distinct orientations in $\cO$ are distinguished by some separation in $N$.
In the case of a submodular universe $\vU$, a separation~$s \in U$ distinguishes a pair of orientations $O_1$ and $O_2$ \emph{efficiently} if it distinguishes them and $O_1$ and $O_2$ cannot be distinguished by any separation of lower order.

Let $\cF \subseteq 2^{\vS}$ be a set of subsets of $\vS$. 
An orientation $P$ of $S$ is an \emph{$\cF$-tangle (of $S$)} if $P$ is consistent and does not contain any element of~$\cF$ as a subset.

If $\vS$ is a separation system inside some universe $\vU$, we call a consistent orientation $P$ of $S$ a \emph{profile of $S$} if it satisfies that
\[
\text{for all $\vr, \vs \in P$ the separation $(\vr \vee \vs)^*$ does not lie in $P$.}
\]

\noindent A profile of $S$ which contains no co-small separation is \emph{regular}.

We can describe the profiles of $S$ as $\cF$-tangles as follows. Set $\cP_S := \{\{\vr, \vs, (\vr \vee \vs)^*\} : \vr, \vs \in \vS\} \cap 2^{\vS}$,
then the set of all $\cP_S$-tangles is exactly the set of all profiles of $S$. We say that $\cF$ is \emph{profile-respecting (for~$\vS$)} if every $\cF$-tangle of $S$ is a profile of~$S$.

In this paper, the profiles which we consider will typically be $\cF$-tangles for a standard and profile-respecting set of stars which contains $\{\rv\}$ for every small $\vr \in \vS$, and thus these profiles will be regular.

\subsection{Tree sets and \texorpdfstring{\mathintitle{S}}{S}-trees}

A \emph{tree set} $N$ is a nested set of unoriented separations without degenerate or trivial separations. 
A star $\sigma \subseteq \vN$ is called a \emph{node of} $N$ if there is a consistent orientation $O$ of $N$ such that $\sigma$ is the set of maximal elements in $O$ (note that in \cite{TreeSets} the nodes of $N$ are called `splitting stars'). 

Let $\vS$ be again a separation system. An \emph{$S$-tree} is a pair $(T,\alpha)$ of a (graph-theoretical) tree $T$ and a map~$\alpha: \vE(T) \rightarrow \vS$ from the oriented edges $\vE(T) := \{\ve, \ev : e \in E(T)\}$ of $T$ to $\vS$
such that $\alpha(\ev) = \sv$ if $\alpha(\ve) = \vs$. 
If $x \in V(T)$ is a leaf of $T$ and $t \in V(T)$ its unique neighbour, then we call $\alpha(x,t) \in \vS$ a \emph{leaf separation (of $T$)}.

An $S$-tree $(T, \alpha)$ is \emph{over} a set $\cF \subseteq 2^{\vS}$ if $\{\alpha(t',t) : (t',t) \in \vE(T)\} \in \cF$ for every node $t \in V(T)$. If $\cF$ is a set of stars in $\vS$, then we call $\sigma_t := \{\alpha(t',t) : (t',t) \in \vE(T)\} \subseteq \vS$ for a node $t \in V(T)$ the \emph{star associated with $t$ (in $T$)}.

An $S$-tree $(T, \alpha)$ is called \emph{irredundant} if there is no node $t \in V(T)$ with two neighbours $t', t'' \in V(T)$ such that $\alpha(t, t') = \alpha(t, t'')$. If $(T, \alpha)$ is an irredundant $S$-tree over a set of stars, then $\alpha$ preserves the natural ordering on $\vE(T)$, i.e.\ $\ve \leq \vf$ if and only if $\alpha(\ve) \leq \alpha(\vf)$, where $\ve = (v_1, v_2) \leq \vf = (w_1, w_2)$ if and only if the unique path in $T$ from $\{v_1, v_2\}$ to $\{w_1, w_2\}$ starts in $v_2$ and ends in $w_1$ \cite{TangleTreeAbstract}*{Lemma 2.1}.

In this paper, we will only consider $S$-trees over sets of stars in $\vS$. 
Furthermore, we will always assume, without explicitly stating it, that all the considered $S$-trees are irredundant. This is no restriction: for every $S$-tree $(T, \alpha)$ over a set $\cF$ of stars there is a subtree $T'$ of $T$ such that $(T', \alpha\hspace{-1mm}\restriction_{\vE(T')})$ is an irredundant~$S$-tree over $\cF$ \cite{TangleTreeAbstract}*{Lemma 2.3}.

Every $S$-tree over a set of stars induces a tree set $\vN := \text{im}(\alpha)$ via $\alpha$ if no $\alpha(\ve)$ is trivial or degenerate for some edge $\ve \in \vE(T)$.
We then say that $N$ \emph{corresponds} to that $S$-tree. On the other hand, if $N$ is a \emph{regular} tree set in $S$, i.e.\ $\vN$ does not contain any small separations, then one can obtain an $S$-tree $(T,\alpha)$ from $N$ as follows. We take the set of all splitting stars of $\vN$ as the vertex set of $T$ and $N$ as the edges of $T$ where a separation $s \in N$ is incident to the two unique splitting stars of $\vN$ that contain $\vs$ and $\sv$, respectively: 
\begin{THM}{\cite{TreeSets}*{Theorem 6.9}}\label{thm:FromTreeSetsToSTrees}
    Let $\vS$ be a separation system and $N \subseteq S$ a regular tree set. Then there exists an $S$-tree $(T, \alpha)$ with $\text{\emph{im}}(\alpha) = \vN$ such that the stars associated with nodes of $T$ are precisely the nodes of $\vN$.
\end{THM}
\noindent This motivates the name `nodes' for the splitting stars of $N$. It is shown in \cite{TreeSets} that the $S$-tree from \cref{thm:FromTreeSetsToSTrees} is unique up to isomorphisms.
Therefore, we say that $T$ is \emph{the} $S$-tree corresponding to $N$.

We say that a consistent orientation $O$ of $S$ \emph{lives} at a node $\sigma$ of a tree set $N \subseteq S$ (or equivalently $\sigma$ \emph{is home} to $O$) if $\sigma \subseteq O$. Similarly, we say that $O$ \emph{lives} at a node $t \in V(T)$ of an $S$-tree $(T, \alpha)$ if $\sigma_t \subseteq O$.
It is easy to see that every consistent orientation of $S$ lives at a (unique) node of any regular tree set $N$.

Given a set of consistent orientations $\cO$ of $S$ and a tree set $N$ that distinguishes $\cO$, we call a node of $N$ \emph{essential (for $\cO$)} if there is an orientation in $\cO$ which lives at that node and otherwise \emph{inessential (for $\cO$)}.

An \emph{isomorphism} between two separation systems $\vS$ and $\vS'$ is a bijection $\varphi: \vS \rightarrow \vS'$ which is order-respecting and commutes with their involutions. 
We call a nested set $N := N(\vS, \cP) \subseteq S$ that distinguishes some set $\cP$ of profiles of $S$ \emph{canonical} if it is invariant under isomorphisms of separation systems, i.e.\ if the map $(S,\cP) \mapsto N$ commutes with all isomorphisms $\phi: \vS \rightarrow \vS'$ so that $\phi(N(\vS, \cP)) = N(\vS', \phi(\cP))$ with $\phi(\cP) := \{\phi(P) : P \in \cP\}$.
Note that if the set $\cP$ is fixed under automorphisms of $\vS$, i.e.\ $\phi(\cP) = \cP$ for all automorphisms $\phi$ of $\vS$, and $N$ is canonical, then this implies that $\phi(N) = N$ for all automorphisms $\phi$ of $\vS$.

We will need the following canonical version of the tree-of-tangles theorem:

\begin{THM}{\cite{CanonicalToTSubmodular}*{Theorem 2}}\label{thm:Tc}
	Let $\vS$ be a submodular separation system and $\cP$ a set of profiles of $S$. Then there is a nested set $N = N(\vS,\cP) \subseteq S$ which distinguishes $\cP$.
	
	This $N(\vS,\cP)$ can be chosen canonically: if $\varphi: \vS \rightarrow \vS'$ is an isomorphism of separation systems and $\cP' := \{\varphi(P) : P \in \cP\}$, then $\varphi(N(\vS,\cP)) = N(\vS',\cP')$.
\end{THM}

\subsection{Tangle-tree duality}

Let $\vS$ be a separation system inside some universe, and let $\vr$ be a separation in $\vS$ which is neither degenerate nor trivial, and set $S_{\geq \vr} := \{x \in S : \vx \geq \vr \text{ or } \xv \geq \vr\}$. We say that \emph{$\vs$ emulates $\vr$ in $\vS$} if $\vs \geq \vr$ and for every $\vx \in \vS_{\geq \vr} \setminus \{\rv\}$ with $\vr \leq \vx$ it holds that $\vs \vee \vx \in \vS$.
We can then define a function $\shiftingfct{\vs}{\vr} : \vS_{\geq \vr} \setminus \{\rv\} \rightarrow \vS_{\geq \vr} \setminus \{\rv\}$ by
\[
\shiftingfct{\vs}{\vr}(\vx) := \vx \vee \vs  \text{ and } \shiftingfct{\vs}{\vr}(\xv) := (\vx \vee \vs)^* \text{ for } \vx \geq \vr.
\]
If $(T,\alpha)$ is an $S$-tree and $\vs$ emulates $\vr$, we can define $\alpha' := \shiftingfct{\vs}{\vr} \circ \alpha$. It is then easy to see that $(T, \alpha')$ is again an $S$-tree, and we call $(T, \alpha')$ the \emph{shift of $(T, \alpha)$ onto $\vs$}.

Given a set $\cF \subseteq 2^{\vS}$ of stars, we say that \emph{$\vs$ emulates $\vr$ in $\vS$ for $\cF$} if $\vs$ emulates $\vr$ in $\vS$ and for every star~$\sigma \in \cF$ with $\sigma \subseteq \vS_{\geq \vr} \setminus \{\rv\}$ that contains an element $\vt \geq \vr$ it holds that $\shiftingfct{\vs}{\vr}(\sigma) \in \cF$.

This property is crucial to the following lemma, which is a key tool in the proof of the tangle-tree duality theorem and which will also play an important role in the proofs of \cref{cor:RefiningACanonicalToT} and \ref{thm:RefiningGoodToTs}.

\begin{LEM}{\cite{TangleTreeAbstract}*{Lemma 4.2}}\label{lem:shifting}
	Let $\vS$ be a separation system, $\cF \subseteq 2^{\vS}$ a set of stars, and let $(T,\alpha)$ be a tight and irredundant $S$-tree over $\cF$. Further, let $\vr$ be a leaf separation of $(T, \alpha)$ which is neither trivial nor degenerate, let $\vs \in \vS$ emulate $\vr$ in $\vS$ for $\cF$, and consider $\alpha' := \shiftingfct{\vs}{\vr} \circ \alpha$. Then $(T,\alpha')$ is an $S$-tree over~$\cF \cup \{\{\sv\}\}$, in which $\{\sv\}$ is a star associated with a unique leaf.
\end{LEM}

\noindent\emph{Tight} here means that for every node $t \in V(T)$ the star associated with $t$ does not contain the inverse of any of its non-degenerate elements. 

It is shown in \cite{TangleTreeAbstract}*{Lemma 2.4} that if there is an $S$-tree over $\cF$ with some set of leaf separations $\vr_i$ which are neither trivial nor degenerate, then there also exists an $S$-tree over $\cF$ which is tight and irredundant such that each $\vr_i$ is a leaf separation and not the image of any other edge.

A set $\cF$ of stars in $\vS$ is \emph{closed under shifting} if whenever $\vs \in \vS$ emulates some $\vr \in \vS$, then it also emulates $\vr$ in $\vS$ for $\cF$.
With this definition, we can now to state the \emph{tangle-tree duality theorem} in the version which we need later:
\begin{THM}{\cite{AbstractTangles}*{Theorem 4}}\label{thm:TTD}
	Let $\vU$ be a universe of separations and $\vS \subseteq \vU$ a submodular separation system. Let $\cF \subseteq 2^{\vS}$ be a set of stars which is standard for $\vS$ and closed under shifting. Then exactly one of the following holds:
	\begin{itemize}
		\item[(i)] There exists an $\cF$-tangle of $S$.
		\item[(ii)] There exists an $S$-tree over $\cF$.
	\end{itemize}
\end{THM}

The conditions of \cref{thm:TTD} may seem to be rather strong at first, but in practice they can typically be satisfied.
Diestel, Eberenz and Erde \cite{ProfileDuality} showed that any set $\cF \subseteq 2^{\vS}$ can be transformed into a standard set $\hat{\cF}$ of stars which is closed under shifting so that an orientation of $S$ is a regular $\hat{\cF}$-tangle if and only if it is a regular $\cF$-tangle:

\begin{LEM}{\cite{ProfileDuality}*{Lemma 11 \& 14}}\label{lem:MakingASetFFriendly}
    Let $\vS$ be a submodular separation system inside some universe and let~$\cF \subseteq 2^{\vS}$ be a standard set. Then there exists a standard set $\hat{\cF} \subseteq 2^{\vS}$ of stars that is closed under shifting such that an orientation of $S$ is a regular $\cF$-tangle if and only if it is a regular $\hat{\cF}$-tangle.
\end{LEM}

Especially, \cref{lem:MakingASetFFriendly} implies that the set $\cP_S$ can be transformed into a set $\hat{\cP}_S$ of stars which is closed under shifting such that the set of all $\hat{\cP}_S$-tangles of $S$ is precisely the set of all profiles of $S$.

\section{Refining inessential stars}\label{sec:RefiningInessStars}

In this section we prove \cref{thm:RefiningGoodToTs}, which is a generalization of Erde's work \cite{JoshRefining} to abstract separation systems. Let $\vS$ be a submodular separation system and $\cF$ some suitable set of stars in $\vS$. Then \cref{thm:RefiningGoodToTs} asserts that every tangle-distinguishing tree set of `good' separations can be refined so that each of its nodes is either a star in $\cF$ or home to an $\cF$-tangle of $S$.

The main tool in the proof of \ref{thm:ErdesRefThm} \cite{JoshRefining} is the so-called \emph{refining lemma} \cite{JoshRefining}*{Lemma~3.1}.
It asserts that given a submodular universe of separations $\vU$, an integer $k > 0$, and a suitable set $\cF$ of stars in $\vS_k$, there is, for every inessential star $\sigma \subseteq \vS_k$, an $S_k$-tree over $\cF \cup \{\{\sv\} : \vs \in \sigma\}$ -- as long as every separation in $\sigma$ distinguishes a pair of $\cF$-tangles of $S_k$ efficiently. 

Such an $S_k$-tree corresponds to a nested set of separations with the property that every node is either a star in $\cF$ or home to an $\cF$-tangle, where the latter case only occurs for singleton stars $\{\sv\}$ with $\vs \in \sigma$. Applying this to the inessential nodes of any canonical tree set $\tilde{N}$ that distinguishes the set of $\cF$-tangles efficiently, and joining those sets with the tree set $\tilde{N}$, then yields \ref{thm:ErdesRefThm}.

In this section, we will first generalize Erde's refining lemma to arbitrary submodular separation systems which will then already directly imply \cref{thm:RefiningGoodToTs}. 
\bigskip

The tangles which we consider in this section will always be $\cF$-tangles for a set $\cF$ of stars in $\vS$ which is standard, closed under shifting and contains $\{\rv\}$ for every small separation $\vr \in \vS$. If $\cF$ is additionally profile-respecting, then we call $\cF$ \emph{friendly}.
Note that requiring $\cF$ to be a friendly set of stars is not really a restriction. Indeed, if $\cF \subseteq 2^{\vS}$ is just \emph{any} set of sets, which not necessarily consists of stars, then by \cref{lem:MakingASetFFriendly}, we can turn $\cF$ into a standard set of stars $\hat{\cF}$ that is closed under shifting so that the set of all $\cF$-tangles is precisely the set of all $\hat{\cF}$-tangles. 
If it is not already the case that every $\cF$-tangle is a profile of $S$, then we can add $\hat{\cP}_S$ to $\hat{\cF}$ to obtain a friendly set of stars $\tilde{\cF} := \hat{\cF} \cup \hat{\cP}_S$ such that an orientation of $S$ is an $\tilde{\cF}$-tangle if and only if it is a profile and an $\cF$-tangle of $S$.

Finally, note that if $G$ is a graph and $\cT_k$ the set of sets whose exclusion defines $k$-tangles as such (see \cite{DiestelBook16noEE}*{Ch.\ 12.5} for a detailed introduction into tangles in graphs), then it is easy to see that the set $\cT_k^* \subseteq \cT_k$ of all stars in $\cT_k$ is a friendly set of stars. Further, Diestel and Oum have shown that the $\cT_k^*$-tangles in $G$ of order $k$ are precisely the $k$-tangles in $G$ if $\abs{G} \geq k$ \cite{TangleTreeGraphsMatroids}*{Lemma 4.2}.
\bigskip

We first recall the exact statement of Erde's refining lemma:

\begin{LEM}{\cite{JoshRefining}*{Lemma~3.1}}\label{lem:JELemma}
	Let $\vU$ be a submodular universe of separations, let $k \in \mathbb{N}$, and let $\cF$ be a friendly set of stars in $\vS_k$.
	Further, let $\sigma = \{\vs_1, \ldots, \vs_n\} \subseteq \vS_k$ be a non-empty star of separations which is inessential for the set of all $\cF$-tangles of $S_k$, and suppose that each $s_i$ distinguishes some pair of $\cF$-tangles of $S_k$ efficiently. 
	Then there is an $S_k$-tree over $\cF' := \cF \cup \{\{\sv_1\}, \ldots, \{\sv_n\}\}$ in which each $\vs_i$ appears as a leaf separation.
\end{LEM}

Besides the fact that the separation system at hand is of the form $S_k$, it is essential in the proof of \cref{lem:JELemma} that every separation in the inessential star $\sigma$ which we want to refine distinguishes a pair of $\cF$-tangles of $S_k$ \emph{efficiently}.
Indeed, Erde \cite{JoshRefining} gave an example of a graph $G$, an integer $k > 0$, and a star~$\sigma$ which is inessential for the set of all $k$-tangles in $G$, and whose separations each distinguish some pair of $k$-tangles in $G$ -- but do not do so efficiently -- such that it is not possible to refine $\sigma$ in the sense of \cref{lem:JELemma}. Moreover, the star $\sigma$ in this example is invariant under all automorphisms of $G$. 

Thus, when generalizing \cref{lem:JELemma} to arbitrary submodular separations systems $\vS$, we cannot hope to refine every inessential star without imposing further conditions on the separations inside that star.
But as the notion of efficiency relies on the existence of an order function, we have to find a different condition on the separations in $\sigma$ which is defined even for abstract separation systems without an order function.

By taking a closer look at the proof of \cref{lem:JELemma} \cite{JoshRefining}, one can see that the efficiency condition on the separations in $\sigma$ is only used to show that for every separation $\vs_i \in \sigma$, there is an $\cF$-tangle $P_i$ of $S_k$ such that $\vs_i$ emulates every separation $\vx \leq \vs_i$ with $\xv \in P_i$. This is equivalent to having $\sv_i \wedge \rv \in \vS$ for every separation $\rv \in \vS$ that is smaller than some $\xv \in P_i$, which motivates the following definition.

Let $\vS$ be a separation system inside some universe, and let $O$ be a consistent orientation of $S$. We say that a separation~$\vs \in \vS$ is \emph{closely related to $O$} if $\vs \in O$ and
\[
\text{for every separation } \vr \in P, \text{ it holds that } \vs \wedge \vr \in \vS.
\]
We remark that the inverse $\sv$ of every separation $\vs \in O$ which is closely related to $O$ indeed emulates every separation $\xv$ with $\vs \leq \vx \in O$.

A separation $\vr \in O$ is \emph{maximal in $O$} if it is a maximal element in $O$ with respect to the partial order on~$O$ induced by $\vS$. An example of separations which are closely related to a profile $P$ are those which are maximal in~$P$, as the following proposition shows:

\begin{PROP}\label{prop:MaxSepsAreCloselyRelated}
    Let $\vS$ be a submodular separation system, and let $P$ a profile of $S$. Then every maximal separation $\vs \in P$ is closely related to $P$.
\end{PROP}

\begin{proof}
    Suppose that there is a separation $\vr \in P$ with $\vs \wedge \vr \notin \vS$. By the submodularity of $\vS$, we then have $\vs \vee \vr \in \vS$, and hence $\vs \vee \vr \in P$ as $P$ is a profile. But $\vs \vee \vr$ is strictly larger than $\vs$, a contradiction.
\end{proof}

Our next basic observation describes when a separation $\vr$ is closely related to some profile $P$ provided that there is a separation which is greater than $\vr$ and closely related to~$P$:

\begin{PROP}\label{prop:ShowingThatASepIsCloselyRelated}
Let $\vS$ be a submodular separation system. If $\vs \in \vS$ is closely related to a profile $P$ of $S$ and $\vr \leq \vs$ is a separation in $\vS$ such that $\vr \wedge \vu \in \vS$ for every $\vu \leq \vs$, then $\vr$ is closely related to $P$.
\end{PROP}
\begin{proof}
    Let $\vt \in P$ be an arbitrary separation. Since $\vs$ is closely related to $P$, we have $\vs \wedge \vt \in \vS$. As $\vr \wedge \vs = \vr$, it follows that $\vr \wedge \vt = (\vr \wedge \vs) \wedge \vt = \vr \wedge (\vs \wedge \vt) \in \vS$ because $\vs \wedge \vt \leq \vs$.
\end{proof}

In general, there are more separations that are closely related to some profile than just the maximal ones. Moreover, being closely related to some profile is indeed just a generalization of efficiently distinguishing a pair of profiles in that if $k$ is a positive integer and $\vU$ is a submodular universe, then the orientations of every separation in~$S_k \subseteq U$ which efficiently distinguishes a pair of profiles of $S_k$ are closely related to the profiles in that pair:

\begin{PROP}\label{prop:eff}
	Let $\vU$ be a submodular universe of separations, $k \in \N$, and let $P$ and $P'$ be two profiles of~$S_k$. If a separation $\vs \in P$ distinguishes $P$ and $P'$ efficiently, then $\vs$ and $\sv$ are closely related to $P$ and~$P'$, respectively.
\end{PROP}

\begin{proof}
	Let $\vs \in P$ be a separation that distinguishes $P$ and $P'$ efficiently, and let $\vr \in P$ be arbitrary. 
	We need to show that $\vr \wedge \vs \in \vS_k$. Suppose for a contradiction that $\abs{\vr \wedge \vs} \geq k > \abs{\vr}$. By submodularity of the order function, it follows that $\abs{\vr \vee \vs} < \abs{\vs} < k$, and hence $\vr \vee \vs \in \vS_k$. Since $P$ is a profile, we then have~$\vr \vee \vs \in P$. Moreover, by the consistency of $P'$, it holds that $(\vr \vee \vs)^* = \rv \wedge \sv \leq \sv \in P'$, and hence~$\vr \vee \vs$ distinguishes $P$ and $P'$. Since $\abs{\vr \vee \vs} < \abs{s}$, this contradicts that $s$ distinguishes $P$ and $P'$ efficiently.
	Analogously, one can show that $\sv$ is closely related to $P'$.
\end{proof}

We are now ready to state our generalized version of Erde's refining lemma (\cref{lem:JELemma}):

\begin{LEM}\label{lem:reflem}
	Let $\vS$ be a submodular separation system, and let $\cF$ be a standard set of stars in $\vS$ which is closed under shifting and contains $\{\rv\}$ for every small $\vr \in \vS$.
	Further, let~$\sigma = \{\vs_1, ..., \vs_n\} \subseteq \vS$ be a non-empty star which is inessential for the set of all $\cF$-tangles of $S$, and suppose that each $\sv_i$ is closely related to some $\cF$-tangle of $S$. 
	Then there is an $S$-tree over $\cF \cup \{\{\sv_1\}, ...,\{\sv_n\}\}$ in which each $\vs_i$ appears as a leaf separation. 
\end{LEM}

The proof of this lemma will be quite similar to Erde's proof of \cref{lem:JELemma} \cite{JoshRefining}. Indeed, we will only need to make some small changes to adjust the proof to our slightly different setting. For the sake of completeness, we give a full proof here instead of just explaining the differences. Any interested reader may also consult \cite{JoshRefining} for the original proof.

\begin{proof}[Proof of \cref{lem:reflem}]
	Set $\bar{\cF} := \cF \cup \{\{\xv\} : \sv_i \leq \xv \text{ for some } i \in [n]\}$. We first show that $\bar{\cF}$ is closed under shifting. For this, let $\{\xv\} \in \bar{\cF} \setminus \cF$ be arbitrary. Then any map $\shiftingfct{\vs}{\vr}$ with $\{\xv\} \subseteq \vS_{\geq \vr}$ sends $\xv$ to some~$\yv \geq \xv$ and so $\{\yv\} \in \bar{\cF}$. Therefore, $\bar{\cF}$ is closed under shifting since $\cF$ was. 
	Further, $\bar{\cF}$ is standard since $\cF$ was. Hence, we can apply \cref{thm:TTD} to $\bar{\cF}$.
	Since $\sigma$ was an inessential star for $\cF$ and $\cF \subseteq \bar{\cF}$, there cannot be an $\bar{\cF}$-tangle of $S$, and hence, by \cref{thm:TTD}, there is an~$S$-tree over $\bar{\cF}$. Let $(T, \alpha)$ be such a tree which contains as many of the separations $\vs_i$ among its leaf separations as possible. Note that, in general, an~$S$-tree over $\bar{\cF}$ need not contain any $\vs_i$ as a leaf separation.

    \begin{figure}[h!]
		\centering
        \definecolor{lgreen}{rgb}{0,0.8,0}
\definecolor{lgrey}{rgb}{0.5,0.5,0.5}
\scalebox{0.8}{%
\begin{tikzpicture}
\draw [line width=1.3pt,color=lgrey] (4.63,8.78)-- (6.61,7.13)-- (7.63,4.97);;
\draw [line width=1.3pt,color=lgrey] (8.18,6.81)-- (7.62,7.87);
\draw [line width=1.3pt,color=lgrey] (6.63,4.23)-- (7.12,3.38)-- (6.04,3.13)-- (3.88,1.59)-- (4.94,3.54);
\draw [line width=1.3pt,color=lgrey] (3.88,1.59)-- (2.25,0.8);
\draw [line width=1.3pt,color=lgrey] (6.63,4.23)-- (7.63,4.97)-- (8.18,6.81)-- (9.8,8.07);
\draw [line width=1.3pt,color=lgrey] (6.63,4.23)-- (4.77,4.78)-- (4.5,5.81);
\draw [line width=1.3pt,color=lgrey] (4.77,4.78)-- (2.86,4.68);
\draw [line width=1.3pt,color=lgrey] (7.12,3.38)-- (9.54,1.81)-- (11.4,1.95);
\draw [line width=1.3pt,color=lgrey] (7.91,2.87)-- (10.45,3.49);

\draw [rotate around={0:(7,6)},line width=1.5pt,color=blue] (7,6) ellipse (2.01cm and 0.22cm);
\draw [rotate around={-37.18:(4.96,2.27)},line width=1.5pt,color=blue] (4.96,2.27) ellipse (1.22cm and 0.21cm);
\draw [rotate around={48.67:(9.04,2.9)},line width=1.5pt,color=blue] (9.04,2.9) ellipse (1.48cm and 0.24cm);
\draw [rotate around={64.8:(3.8,4.72)},line width=1.5pt,color=blue] (3.8,4.72) ellipse (0.97cm and 0.21cm);

\draw [rotate around={33.34:(5.6,8.02)},line width=1.5pt] (5.6,8.02) ellipse (1.14cm and 0.23cm);
\draw [rotate around={-50.47:(9.05,7.61)},line width=1.5pt] (9.05,7.61) ellipse (1.11cm and 0.19cm);
\draw [rotate around={67.63:(10.64,2.11)},line width=1.5pt] (10.64,2.11) ellipse (0.98cm and 0.2cm);
\draw [rotate around={-47.22:(3.05,1.21)},line width=1.5pt] (3.05,1.21) ellipse (0.94cm and 0.19cm);

\draw [color=blue](9.6,5.6) node[anchor=north west] {\LARGE $\sigma$};
\draw [color=lgreen](3.55,8.9) node[anchor=north west] {\LARGE $P_i$};
\draw [color=blue](5.8,5.65) node[anchor=north west] {\Large $\vs_i$};
\draw (4.95,7.3) node[anchor=north west] {\Large $\vx_i$};
\draw [color=lgrey](10.3,8) node[anchor=north west] {\LARGE $(T, \alpha)$};

\draw [->,line width=1.3pt,color=blue] (6.5,5.79) -- (6.5,4.95);
\draw [->,line width=1.3pt,color=blue] (3.83,4.33) -- (4.51,4.01);
\draw [->,line width=1.3pt,color=blue] (5.6,2.01) -- (6.07,2.61);
\draw [->,line width=1.3pt,color=blue] (9.22,3.45) -- (8.6,3.9);

\draw [->,line width=1.3pt] (3,1.56) -- (3.5,2);
\draw [->,line width=1.3pt] (5.39,7.63) -- (5.9,7);
\draw [->,line width=1.3pt] (9.24,7.11) -- (8.75,6.7);
\draw [->,line width=1.3pt] (10.54,2.39) -- (9.9,2.7);

\begin{scriptsize}
\draw [fill=lgrey] (4.63,8.78) circle (2pt);
\draw [fill=lgrey] (6.61,7.13) circle (2pt);
\draw [fill=lgrey] (6.63,4.23) circle (2pt);
\draw [fill=lgrey] (8.18,6.81) circle (2pt);
\draw [fill=lgrey] (9.8,8.07) circle (2pt);
\draw [fill=lgrey] (7.62,7.87) circle (2pt);
\draw [fill=lgrey] (7.12,3.38) circle (2pt);
\draw [fill=lgrey] (6.04,3.13) circle (2pt);
\draw [fill=lgrey] (3.88,1.59) circle (2pt);
\draw [fill=lgrey] (4.94,3.54) circle (2pt);
\draw [fill=lgrey] (2.25,0.8) circle (2pt);
\draw [fill=lgrey] (7.63,4.97) circle (2pt);
\draw [fill=lgrey] (4.77,4.78) circle (2pt);
\draw [fill=lgrey] (4.5,5.81) circle (2pt);
\draw [fill=lgrey] (2.86,4.68) circle (2pt);
\draw [fill=lgrey] (9.54,1.81) circle (2pt);
\draw [fill=lgrey] (11.4,1.95) circle (2pt);
\draw [fill=lgrey] (7.91,2.87) circle (2pt);
\draw [fill=lgrey] (10.45,3.49) circle (2pt);
\draw [fill=lgreen] (4.19,8.18) circle (3pt);
\end{scriptsize}
\end{tikzpicture}
}%
		\caption{The $\vS_k(G)$-tree over $\bar{\cF}$; the labeled leafs correspond to separations that are forced by $\bar{\cF}$ but not by $\cF$.}
		\label{fig:Reflem}
	\end{figure}    
	
	Suppose there is an $\vs_i \in \sigma$ which is not a leaf separation of $(T, \alpha)$, and let $P_i$ be some $\cF$-tangle of~$S$ to which~$\sv_i$ is closely related. Since $P_i$ is a consistent orientation of $S$, it has to live at some node $t \in V(T)$. But since $P_i$ is an $\cF$-tangle, the star which is associated with $t$ can only be in $\bar{\cF} \setminus \cF$, and hence $t$ has to be a leaf. Therefore, there is some leaf separation $\vx_i$ of $(T,\alpha)$ such that $\xv_i \in P_i$. Since $\{\xv_i\} \in \bar{\cF} \setminus \cF$, it follows that~$\xv_i \geq \sv_j$ for some $j \in [n]$. 
	
	But as $\cF$ contains $\{\rv\}$ for every small $\vr \in \vS$, we have that $P_i$ is regular, and so $\sv_i$ is the unique separation in $\sigma$ with $\sv_i \in P_i$.
	Indeed, if $\sv_j \in P_i$ for some $j \in [n]\setminus\{i\}$, then since~$P_i$ is consistent, we cannot have $\vs_i \leq \sv_j$. So either $\sv_i \leq \sv_j$ or $\sv_j \leq \sv_i$ or $\sv_j \leq \vs_i$. As~$\sigma$ is a star, the first two cases imply that either $\vs_j$ or $\vs_i$ is small, which, in both cases, contradicts that~$P_i$ is regular. Moreover, again because $\sigma$ is a star, the latter case implies that $\vs_i = \sv_j$, which contradicts that $P_i$ is an orientation of~$S$ as then~$\sv_i, \vs_i \in P_i$.
	Hence, $j = i$ and $\sv_i \leq \xv_i$.
	
	Since $\sv_i$ is closely related to $P_i$, we have that $\vs_i$ emulates $\vx_i$. Further, since $\vx_i$ is neither trivial nor degenerate, we can assume, by the comment after \cref{lem:shifting}, that $T$ is tight and irredundant and that~$\vx_i$ is not the image of any other edge in $T$.
	Therefore, by \cref{lem:shifting}, the shift of $(T,\alpha)$ onto $\vs_i$ yields an $S$-tree~$(T,\alpha')$ over $\bar{\cF}$, which  contains $\vs_i$ as a leaf separation and not as the image of any other edge.
	Moreover, since $\vs_i < \sv_j$ for all $j \neq i$, the shift $(T, \alpha')$ also contains all separations $\vs_j$ as leaf separations that $T$ did and hence contains one more than $T$, contradicting the choice of $T$. Therefore, $(T, \alpha)$ contains all $\vs_i$ as leaf separations.
	
	We are left to show that $(T, \alpha)$ is an $S$-tree over $\cF' := \cF \cup \{\{\sv_1\}, \ldots, \{\sv_n\}\}$. For this, recall that by definition, $(T, \alpha)$ is already an $S$-tree over $\bar{\cF}$. So suppose there is a node $t \in V(T)$ whose associated star under $\alpha'$ is not in~$\cF'$. Since $(T, \alpha)$ is over $\bar{\cF}$ and $\cF' \setminus \bar{\cF}$ contains only singleton stars, it follows that~$t$ is a leaf of~$T$. Let $\vr$ be the leaf separation associated with the unique edge which is incident to $t$. Since $\{\rv\} \in \bar{\cF} \setminus \cF'$, there is some $i \in [n]$ such that $\rv > \sv_i$. 
	But it also holds that $\sv_i > \vr$ (since $\vs_i$ and~$\vr$ are both leaf separations). Hence, $\vr$ is trivial and so $\{\rv\} \in \cF \subseteq \cF'$, which completes the proof.
\end{proof}

Even though \cref{lem:JELemma} and \cref{lem:reflem} are interesting results on their own, the importance of \cref{lem:JELemma} comes from the fact that it can be applied to tree sets $\tilde{N}$, inside some $S_k$, which efficiently distinguish the set of $\cF$-tangles of $S_k$ for some friendly set $\cF$ of stars. It then asserts the existence of a nested set $N \subseteq S_k$ with $\tilde{N} \subseteq N$ in which each inessential node is too `small' to be home to a profile, i.e.\ every inessential node is a star in~$\cF$.

For separation systems of the form $\vS_k$ it is already known that there are tree sets which satisfy the conditions of \cref{lem:JELemma} and \cref{lem:reflem}. 
For instance, Carmesin, Diestel, Hamann and Hundertmark~\cite{CDHH13CanonicalAlg} constructed a canonical tree set for graphs which efficiently distinguishes a given set of $k$-tangles. 
More generally, Diestel, Hundertmark and Lemanczyk \cite{ProfilesNew} constructed a canonical tree set distinguishing all `robust' $k$-profiles efficiently, for some $k \in \N$, inside a submodular universe of separations. Applying \cref{lem:JELemma} to these tree sets then yields \ref{thm:ErdesRefThm}.

But if we want to apply \cref{lem:reflem} to some canonical tree set in an arbitrary submodular separation system to obtain a similar result as in \ref{thm:ErdesRefThm}, we first have to find such a set that satisfies the assumptions of \cref{lem:reflem}. 
For this, we will show in \cref{subsec:RefiningTheCanonicalToTFromEK} that the canonical tree set which Elbracht and Kneip \cite{CanonicalToTSubmodular} constructed satisfies those assumptions. This will then allow us to prove \cref{cor:RefiningACanonicalToT}, by refining this set.

However, \cref{lem:JELemma} is actually strong enough to show that we can in fact refine \emph{every} tree set inside some $S_k$, canonical or not, in which every separation distinguishes some pair of $\cF$-tangles efficiently. 
By \cref{lem:reflem} we can now also refine every tangle-distinguishing tree set $N$ inside an abstract submodular separation system, as long as the inverse of every separation which is contained in an inessential node of~$N$ is closely related to some $\cF$-tangle:

\begin{THM}\label{thm:RefiningToTsWeakAssumptionn}
    Let $\vS$ be a submodular separation system, and let $\cF$ be a standard set of stars in $\vS$ which is closed under shifting and contains $\{\rv\}$ for every small $\vr \in \vS$.
    Further, let $\tilde{N}$ be a nested set of separations that distinguishes all $\cF$-tangles of $S$. If the inverse of every separation that is contained in an inessential node of $\tilde{N}$ is closely related to some $\cF$-tangle of $S$, then there exists a nested set $N \subseteq S$ with $\tilde{N} \subseteq N$ such that every node of~$N$ is either a star in $\cF$ or home to an \mbox{$\cF$-tangle of~$S$.}
\end{THM}

\begin{proof}
    Let $\Sigma$ be the set of all inessential nodes of $\tilde{N}$, and let $\sigma \in \Sigma$ be arbitrary. Since the inverse of every separation in $\sigma$ is closely related to some $\cF$-tangle of $S$, we can apply \cref{lem:reflem} to $\sigma$ to obtain a nested set $N_\sigma$ that corresponds to an $S$-tree over~$\cF \cup \{\{\sv\} : \vs \in \sigma\}$ in which all separations in $\sigma$ appear as leaf separations. Clearly, $N_\sigma$ is nested with~$N_{\sigma'}$ for every star $\sigma' \in \Sigma$.
    Since $N_\sigma$ is also nested with $N$, setting $N := \tilde{N} \cup \bigcup_{\sigma \in \Sigma} N_\sigma$ yields the desired nested set.
\end{proof}

\noindent Note that the nested set from \cite{CanonicalToTSubmodular} satisfies the assumptions of \cref{thm:RefiningToTsWeakAssumptionn}, as we will show in \cref{subsec:RefiningTheCanonicalToTFromEK}.

While the assumption of \cref{thm:RefiningToTsWeakAssumptionn} on the set $\tilde{N}$ is the weakest one we need in order to be able to prove such a refinement theorem for abstract separations systems by applying \cref{lem:reflem}, we will propose in the following a stronger assumption which we believe is more natural. For this, recall that \ref{thm:ErdesRefThm} imposes the condition on $\tilde{N}$ that every separation inside $\tilde{N}$ should distinguish some pair of $\cF$-tangles efficiently. Here, it does not depend on the set $\tilde{N}$ whether a separation efficiently distinguishes some pair of $\cF$-tangles. In particular, when seeking to construct a tangle-distinguishing tree set which satisfies the assumptions of \ref{thm:ErdesRefThm}, it can be determined prior to the construction which separations could potentially be included in $\tilde{N}$.

However, \cref{thm:RefiningToTsWeakAssumptionn} does not have this advantage. As its assumptions require that the inverse of every separation which is contained in an inessential node of $\tilde{N}$ should be closely related to some $\cF$-tangle, one needs to know about the whole set $\tilde{N}$ and first determine its inessential nodes, before one can check the condition of \cref{thm:RefiningToTsWeakAssumptionn}.
To obtain a result which is more in line with \ref{thm:ErdesRefThm}, we are therefore interested in finding a condition on the separations in $\tilde{N}$ that has the following two properties. On the one hand it should ensure that one can refine the inessential nodes of $\tilde{N}$, and on the other hand it should be formulated purely in terms of $S$ and $\cF$ so that this property can be checked without reference to the tree set $\tilde{N}$ and its inessential nodes. 

Let $\vS$ be a submodular separation system and $\cP$ a set of profiles of $S$. We say that a separation $ s \in S$ is \emph{good (for $\cP$)} if there are profiles $P, P' \in \cP$ such that $\vs$ and $\sv$ are closely related to $P$ and $P'$, respectively. 
Note that, by \cref{prop:eff}, every separation in an $S_k$ that distinguishes some pair of profiles efficiently is good. Thus, being good is a structural generalization of efficiency to abstract separation systems which come without the framework needed to define efficiency itself.
Since requiring every separation to be good is clearly a stronger assumption than the one from \cref{thm:RefiningToTsWeakAssumptionn}, the following theorem follows directly:

\begin{customthm}{\ref{thm:RefiningGoodToTs}}
    \emph{Let $\vS$ be a submodular separation system, and let $\cF$ be a standard set of stars in $\vS$ which is closed under shifting and contains $\{\rv\}$ for every small $\vr \in \vS$.
    Further, let $\tilde{N}$ be a nested set of separations that distinguishes all $\cF$-tangles of $S$. If every separation in $\tilde{N}$ is good for the set of all $\cF$-tangles of $S$, then there exists a nested set $N \subseteq S$ with $\tilde{N} \subseteq N$ such that every node of~$N$ is either a star in $\cF$ or home to an $\cF$-tangle of $S$.} \qed
\end{customthm}

Note that this theorem is a generalization of Erde's work in \cite{JoshRefining}. Indeed, Erde's refining lemma (\cref{lem:JELemma}) implies that one can refine every tree set inside some separation system of the form $\vS_k$ which contains only essential separations. Since by \cref{prop:eff} every separation which efficiently distinguishes some pair of profiles in $\cP$ is also good for $\cP$, \cref{thm:RefiningGoodToTs} also implies that these tree sets can be refined. 
However, \cref{thm:RefiningGoodToTs} can also be applied to arbitrary submodular separation systems which are not of the form $\vS_k$.

In \cref{sec:ANearlyCanonicalToTWithGoodSeps} we will construct a nested set which distinguishes a given set of profiles and contains only separations that are good for those profiles. Moreover, this set will be canonical, but in a slightly weaker sense than usual.

\section{Refining a canonical tree of tangles}\label{subsec:RefiningTheCanonicalToTFromEK}

In this section we prove that the inessential nodes of the canonical nested set $N(\vS, \cP)$ from \cref{thm:Tc} satisfy the conditions of \cref{lem:reflem}. This will allow us to refine its inessential nodes which then implies \cref{cor:RefiningACanonicalToT}. For this, throughout this section, let $\vS$ be a submodular separation system, and let $\cP$ be a set of profiles of $S$. We first recall the construction of $N(\vS, \cP)$ from \cite{CanonicalToTSubmodular}:

\begin{CONSTR}\label{constr:Tc}
A separation $\vs \in \vS$ is \emph{exclusive (for $\cP$)} if it is contained in exactly one profile in $\cP$. If $P \in \cP$ is that profile, we say that $\vs$ is \emph{$P$-exclusive (for $\cP$)}.
Further, we denote, for every profile $P \in \cP$, the set of all maximal $P$-exclusive separations in $\vS$ with $M_P$. Note that, in general, there will be profiles for which $M_P$ is empty.
If $M_P$ is non-empty for some $P \in \cP$, then the infimum of $M_P$ exists and is again $P$-exclusive for $\cP$ \cite{CanonicalToTSubmodular}*{Lemma 4.4}. Moreover, for every pair of profiles $P, P' \in \cP$, the infima of $M_P$ and~$M_{P'}$ are nested (\cite{CanonicalToTSubmodular}*{Lemma 4.3} and \cref{lem:Fischlemma}).

The strategy for constructing $N := N(\vS, \cP)$ is to pick, for every profile $P \in \cP_1 := \{P \in \cP : M_P \neq \emptyset\}$,
the infimum of $M_P$ and add those separations to $N$, which will then be a nested set. 
As the infima of all non-empty $M_P$ are still exclusive for $\cP$ and hence distinguish every profile in $\cP_1$ from all other profiles in~$\cP$, we can discard the profiles in $\cP_1$ from $\cP$. We then remove from $S$ all separations that are not nested with those infima. Iterating this procedure yields the desired tree set.

To make the construction more precise, and to establish some notation we need later, set $M_{P,1} := M_P$ for every $P \in \cP_1$, and note that $\cP_1$ is non-empty \cite{CanonicalToTSubmodular}*{Lemma 4.2}.
Set $\vs_P := \text{inf}(M_{P,1}) := \bigwedge_{\vr \in M_{P,1}} \vr$ for all~$P \in \cP_1$, set $N_1 := \{s_P : P \in \cP_1\}$, and let $\vS_2$ be the set of all separations in $\vS$ which are nested with~$N_1$. For consistency set also $\vS_1 := \vS$.
Then $\vS_2$ is still submodular and distinguishes all remaining profiles in~$\cP \setminus \cP_1$ \cite{CanonicalToTSubmodular}*{Lemma 4.5}. Hence, we can proceed in the same manner as follows.

At the beginning of the $i$-th step of the construction we have already constructed a set $N_{i-1}$ that distinguishes every profile in $\cP_{<i} := \bigcup_{j < i} \cP_j$ from all other profiles in $\cP$. We then consider the set $\cP_{\geq i} := \cP \setminus \cP_{<i}$ and the separation system $\vS_i$ that consists of all separations in $\vS$ which are nested with~$N_{i-1}$.
For every $P \in \cP_{\geq i}$ we let $M_{P,i}$ be the set of all separations in $\vS_i$ which are $P$-exclusive for $\cP_{\geq i}$ and maximal in $P \cap \vS_i$. 
We then set $\cP_i := \{P \in \cP_{\geq i} : M_{P,i} \neq \emptyset\}$ and $\vs_P := \text{inf}(M_{P,i})$ for every $P \in \cP_i$. For this note again that $\vs_P \in \vS$ for every $P \in \cP_i$ \cite{CanonicalToTSubmodular}*{Lemma 4.4}. 
Lastly, we set $N_i := \{s_P : P \in \cP_{\leq i}\}$, and we let~$\vS_{i+1}$ be the set of all separations in $\vS$ which are nested with~$N_i$. 
Again we have that $N_i$ is nested, and it follows directly from \cref{lem:Fischlemma} that $\vS_i$ is submodular.

We then obtain the desired nested set by setting $N(\vS, \cP) := N_K$ where $K$ is the smallest number for which~$\cP_{>K}$ contains at most one profile. By construction, $N(\vS, \cP)$ distinguishes all profiles in $\cP$, and it is shown in \cite{CanonicalToTSubmodular} that~$N(\vS, \cP)$ is indeed canonical. This completes the construction.
\end{CONSTR}

\begin{figure}
    \centering
    \definecolor{dgrey}{rgb}{0.25,0.25,0.25}
\definecolor{lgreen}{rgb}{0,0.8,0}
\definecolor{lblue}{rgb}{0.3,0.3,1}
\definecolor{dgreen}{rgb}{0,0.4,0}
\definecolor{dblue}{rgb}{0,0,0.8}
\scalebox{0.5}{%
\begin{tikzpicture}
\draw [line width=1.2pt,color=dblue] (6.96,8.46)-- (10.12,8.46);
\draw [line width=1.2pt,color=dblue] (4.5,7.9)-- (2.8,5.1);
\draw [line width=1.2pt,color=dblue] (11.6,7.6)-- (13.4,5.5);
\draw [line width=1.2pt,color=dblue] (4,3)-- (6.9,1.4);
\draw [line width=1.2pt,color=dblue] (8.1,0.9)-- (11,1.7);
\draw [line width=1.2pt,color=dblue] (12.3,2.5)-- (13.5,4.4);
\draw [line width=1.2pt,color=dgreen] (6.08,7.14)-- (4.1,4.18);
\draw [line width=1.2pt,color=dgreen] (10.6,7.1)-- (11.7,2.3);
\draw [line width=1.2pt,color=dgreen] (4.1,3.6)-- (11.2,2.1);
\draw [line width=1.2pt,color=dgreen] (6.8,7.2)-- (10.1,7.2);

\draw [->,line width=1.2pt,color=dblue] (3.5,6.2) -- (2.9,6.6);
\draw [->,line width=1.2pt,color=dblue] (8,8.5) -- (8,9.3);
\draw [->,line width=1.2pt,color=dblue] (12.2,6.8) -- (12.7,7.3);
\draw [->,line width=1.2pt,color=dblue] (12.7,3.2) -- (13.3,2.8);
\draw [->,line width=1.2pt,color=dblue] (9.1,1.2) -- (9.3,0.5);
\draw [->,line width=1.2pt,color=dblue] (6,1.9) -- (5.6,1.2);
\draw [->,line width=1.2pt,color=dgreen] (4.7,5.1) -- (4.1,5.6);
\draw [->,line width=1.2pt,color=dgreen] (7.7,7.2) -- (7.7,7.9);
\draw [->,line width=1.2pt,color=dgreen] (10.9,5.6) -- (11.7,5.8);
\draw [->,line width=1.2pt,color=dgreen] (8.5,2.7) -- (8.3,1.9);

\draw [color=dgrey](7.7,5.7) node[anchor=north west] {{\Huge $\mathcal{P}_{\geq 3}$}};
\draw [color=lblue](12.3,9.6) node[anchor=north west] {{\Huge $\mathcal{P}_1$}};
\draw [color=lgreen](10.16,8.3) node[anchor=north west] {{\Huge$\mathcal{P}_2$}};
\draw [color=dblue](4.8,9.3) node[anchor=north west] {{\Huge $\mathit{N}_1$}};
\draw [color=dgreen](12, 2) node[anchor=north west] {{\Huge $\mathit{N}_2 \setminus \mathit{N}_1$}};
\draw [color=lblue](8.8,10.3) node[anchor=north west] {{\huge $P$}};
\draw [color=dblue](6.7,9.5) node[anchor=north west] {{\huge $\vec{s}_P$}};
\begin{scriptsize}
\draw [fill=lblue] (8.6,9.5) circle (3pt);
\draw [fill=lblue] (13.4,7) circle (3pt);
\draw [fill=lblue] (14,3) circle (3pt);
\draw [fill=lblue] (10.1,0.5) circle (3pt);
\draw [fill=lblue] (4.8,1.1) circle (3pt);
\draw [fill=lblue] (2.5,7.2) circle (3pt);
\draw [fill=lgreen] (8.5,7.7) circle (3pt);
\draw [fill=lgreen] (4.3,6.1) circle (3pt);
\draw [fill=lgreen] (12.1,5) circle (3pt);
\draw [fill=lgreen] (7.6,2.1) circle (3pt);
\draw [fill=dgrey] (7,4.8) circle (3pt);
\draw [fill=dgrey] (7.6,6.1) circle (3pt);
\draw [fill=dgrey] (9.4,5.5) circle (3pt);
\draw [fill=dgrey] (9.6,3.9) circle (3pt);
\draw [fill=dgrey] (8.2,4.2) circle (3pt);
\draw [fill=dgrey] (5.5,4.7) circle (3pt);
\end{scriptsize}
\end{tikzpicture}
}%
    \caption{A visualization of \cref{constr:Tc}.}
    \label{fig:ConstructionOfN(SP)}
\end{figure}

Before we turn to the proof that every inessential node of $N(\vS, \cP)$ satisfies the assumptions of \cref{lem:reflem}, let us first verify a property of $N(\vS, \cP)$ which \cref{fig:ConstructionOfN(SP)} already indicates. In \cref{fig:ConstructionOfN(SP)} all the separations $\vs_P$ in $\vN$ point `outwards'. In particular, in \cref{fig:ConstructionOfN(SP)} there is no pair of separations $s_{P}, s_{P'}$ such that $\vs_P$ and $\vs_{P'}$ point `towards' each other. The next lemma now claims that \cref{fig:ConstructionOfN(SP)} is actually accurate with respect thereto:

\begin{LEM}\label{lem:SepsInNDontPointTowardsEachOther}
    There is no pair of separations $s_P, s_{P'} \in N(\vS, \cP)$ such that $\vs_P < \sv_{P'}$.
\end{LEM}

\begin{proof}
    Suppose for a contradiction that there are separations $s_P, s_{P'} \in N(\vS, \cP)$ with $\vs_P < \sv_{P'}$. Let $i, j \in [K]$ be the indices with $P \in \cP_i$ and $P' \in \cP_j$, and assume without loss of generality that~$j \leq i$. Then it follows from the $P'$-exclusivity of $\vs_{P'}$ for $\cP_{\geq j}$ that $\sv_{P'} \in P$. Moreover, we have $\vs_{P'} \in N_j$, which implies that $\vs_{P'} \in \vS_i$, and in particular $\sv_{P'} \in \vS_i$.
    As $M_{P,i}$ is the set of all maximal elements in $P \cap \vS_i$ and $\sv_{P'} \in P \cap \vS_i$, either there is a separation $\vr \in M_{P,i}$ which crosses $s_{P'}$, or $\sv_{P'} \leq \vr$ for all $\vr \in M_{P,i}$. Since the latter case implies that~$\sv_{P'} \leq \vs_P$, which contradicts that $\vs_P < \sv_{P'}$, we may assume that there is a separation $\vr \in M_{P,i}$ that crosses $\vs_{P'}$. In particular, it follows that $i \leq j$ and thus $i = j$ since $r \in S_i$ is nested with $N_\ell$ for every $\ell < i$. 
    
    As $\vr$ is maximal in $P \cap \vS_{i}$, we have $\vr \vee \sv_{P'} \notin P \cap \vS_i$, which implies that $\vr \vee \sv_{P'} \notin \vS_i$ because $P \cap \vS_i$ is a profile of $S_i$. Thus, by the submodularity of $\vS_i$, we have $\vr \wedge \sv_{P'} \in \vS_i$. 
    As $\vr$ is $P$-exclusive for $\cP_{\geq i}$ and $P' \in \cP_j \subseteq \cP_{\geq i}$, it follows that $\rv \in P'$, and thus, since $P'$ is a profile, $\rv \vee \vs_{P'} \in P'$.
    By the same argument as above, there has to be a separation $\vr' \in M_{P', i}$ that crosses $\rv \vee \vs_{P'}$, which by \cref{lem:Fischlemma} implies that $r'$ crosses $r$. But this contradicts the maximality of either $\vr \in M_{P,i}$ or $\vr' \in M_{P', i}$ as by the submodularity of $\vS_i$ at least one of~$\vr \vee \rv'$ or $\rv \vee \vr'$ is an element of $\vS_i$ and thus contained in $P \cap \vS_i$ or $P' \cap \vS_i$, respectively.
\end{proof}

We now prove that every inessential node of the tree set $N := N(\vS, \cP)$ from \cref{constr:Tc} satisfies the conditions of \cref{lem:reflem}. Our proof will consist of two steps. First, we show that every $\vs_P \in \vN$ is closely related to~$P$. Then, we show that every separation $\vr \in \vN$ which is contained in an inessential node of $N$ is of the form $\vr = \sv_P$ where $P$ is some profile in $\cP$. This then clearly implies the assertion.

In order to prove that every $\vs_P$ is closely related to $P \in \cP_i$, we first show that the infimum of a set of separations which are all closely related to some profile is again closely related to that profile. It then follows that $\vs_P$ is closely related to $P$ as soon as we can show that $M_{P,i}$ is closely related to it. For this, we prove a more general proposition whose stronger assertion we will need in a following chapter.

\begin{PROP}\label{prop:InfimaOfClRelSetsAreClRel}
    Let $\vs \in \vS$, and let $M \subseteq \vS$ be some set of separations such that every~$\vm \in M$ is closely related to some profile $Q_m$ of $S$ which satisfies that $\vs \in Q_m$. Then $\vs \wedge \inf(M) := \vs \wedge \bigwedge_{\vm \in M} \vm$ is an element of $\vS$. Moreover, if $\vs$ is closely related to some profile $P$ of $S$, then $\vs \wedge \inf(M)$ is closely related to $P$.
\end{PROP}

\begin{proof}
    We proceed by induction on $\abs{M}$. If $\abs{M} = 0$ there is nothing to show, so we may assume that $\abs{M} \geq 1$ and that the assertion holds for all sets $M' \subsetneq M$. Let $\vr \in M$ be arbitrary. By the induction hypothesis, $\vt := \vs \wedge \inf(M\setminus \{\vr\})$ is an element of $\vS$, and it is closely related to~$P$ if $\vs$ is closely related to~$P$. It follows that also $\vs \wedge \inf(M) = \vt \wedge \vr$ is an element of $\vS$ because $\vr$ is closely related to $Q_r$ and $\vt \leq \vs \in Q_r$ by the consistency of $Q_r$. 
    
    Now suppose that $\vs$ is closely related to $P$. Recall that, by induction, $\vt$ is also closely related to $P$. So to prove that $\vs \wedge \inf(M) = \vt \wedge \vr$ is closely related to $P$ it is, by \cref{prop:ShowingThatASepIsCloselyRelated}, enough to show that $\vx \wedge (\vt \wedge \vr) \in \vS$ for all $\vx \leq \vt$. For this, let some $\vx \leq \vt$ be given. Then $\vx \wedge (\vt \wedge \vr) = \vx \wedge \vr \in \vS$ because $\vr$ is closely related to $Q_r$ and $\vx \in Q_r$ by the consistency of $Q_r$.
\end{proof}

The next lemma shows that every separation in $M_{P,i}$ is closely related to $P \in \cP_i$ under the assumption that for every profile $P' \in \cP_{<i}$ the separation $\vs_{P'}$ is closely related to $P'$. It then follows by induction and \cref{prop:InfimaOfClRelSetsAreClRel} that $\vs_P$ is closely related to $P$.

\begin{LEM}\label{lem:NestedMaxSepsAreClRel}
    Let $P$ be a profile of $S$, and let $Y \subseteq P$ be a nested set of separations such that the inverse $\yv$ of every separation $\vy \in Y$ is closely related to some profile $Q_y$ of $S$. Further, let $P_Y \subseteq P$ be the set of all separations in $P$ that are nested with $Y$. Then every maximal separation in $P_Y$ is closely related to $P$.
\end{LEM}

\begin{proof}
	Let $\vr \in P$ be any maximal separation in $P_Y$, and let some $\vx \in P$ be given. We need to show that $\vr \wedge \vx \in \vS$.
	For this, set $Y' := \{\vy \in Y : \vr < \yv\}$ and first assume that there is a separation $\vx' \in P$ which is nested with $Y'$ and satisfies $\vr \wedge \vx' = \vr \wedge \vx$. Then the assertion follows. If $\vr \wedge \vx' \notin \vS$, then it follows from the submodularity of $S$ that $\vr \vee \vx' \in \vS$. Since $P$ is a profile, we then have $\vr \vee \vx' \in P$, which contradicts the maximality of $\vr$ in $P_Y$ as $\vr \vee \vx'$ is nested with $Y$ by \cref{lem:Fischlemma}. 
	
	To conclude the proof it thus suffices to find a separation $\vx'$ as above.
	To this end, we pick a separation~$\vx' \in P$ which is minimal (with respect to the partial order on $\vS$) so that $\vr \wedge \vx' = \vr \wedge \vx$. For this note that $\vx$ itself is a candidate for $\vx'$ (see \cref{fig:NestedMaxSepsAreClRel1}). 

    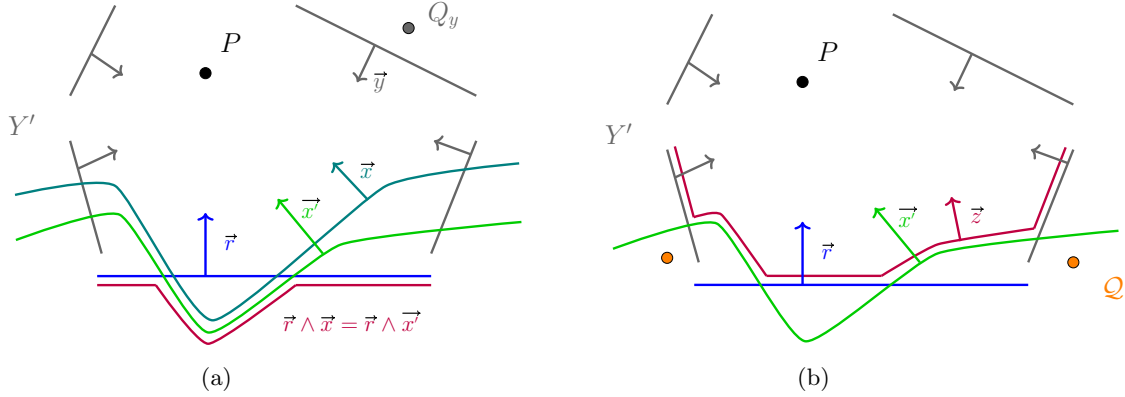
\begin{figure}[h]
        \centering
        \hspace{-15mm}
        \begin{subfigure}[b]{0.35\linewidth}
            \definecolor{lgrey}{rgb}{0.4,0.4,0.4}
\definecolor{ligrey}{rgb}{0.8,0.8,0.8}
\definecolor{lgreen}{rgb}{0,0.8,0}
\scalebox{0.6}{%
\begin{tikzpicture}
\draw [line width=1.5pt,color=lgrey] (5,10)-- (4,8);
\draw [line width=1.5pt,color=lgrey] (4,7)-- (4.7,4.5);
\draw [line width=1.5pt,color=lgrey] (9,10)-- (13,8);
\draw [line width=1.5pt,color=lgrey] (13,7)-- (12,4.5);

\draw [line width=1.5pt, color=blue] (4.6,4)-- (12,4);


\draw [teal, line width=1.5pt] plot [smooth, tension=0.3] coordinates {(2.78,5.8) (5,6) (7.16,3.02) (11,6) (13.95,6.5)};
\draw [lgreen, line width=1.5pt] plot [smooth, tension=0.3] coordinates {(2.8,4.8) (5.1,5.36) (7.06,2.75) (10,4.7) (14,5.2)};

\draw [purple, line width=1.5pt] (4.6, 3.8) -- (5.9, 3.8);
\draw [purple, line width=1.5pt] plot [smooth, tension=0.3] coordinates {(5.9,3.8) (7.06,2.5) (9, 3.8)};
\draw [purple, line width=1.5pt] (9, 3.8) -- (12, 3.8);
\draw (8.6, 3.3) node[purple,anchor=north west] {\LARGE $\vr \wedge \vx = \vr \wedge \vx'$};

\draw [->,line width=1.5pt,color=blue] (7,4) -- (7,5.4);
\draw (7.3,5) node[anchor=north west,color=blue] {\LARGE $\vr$};

\draw [->,line width=1.5pt,color=teal] (10.6,5.69) -- (9.82,6.52);
\draw (10.3,6.6) node[anchor=north west,color=teal] {\LARGE $\vx$};
\draw [->,line width=1.5pt,color=lgreen] (9.6,4.5) -- (8.6,5.71);
\draw (9,5.8) node[anchor=north west,color=lgreen] {\LARGE $\vx'$};

\draw [->,line width=1.5pt,color=lgrey] (4.47,8.93) -- (5.17,8.45);
\draw [->,line width=1.5pt,color=lgrey] (4.17,6.38) -- (5.05,6.79);
\draw [->,line width=1.5pt,color=lgrey] (10.75,9.12) -- (10.36,8.3);
\draw (10.6,8.6) node[anchor=north west,color=lgrey] {\LARGE $\vy$};
\draw [->,line width=1.5pt,color=lgrey] (12.88,6.7) -- (12.06,7);


\draw (7.2,9.5) node[anchor=north west] {\huge $P$};
\draw (11.8,10.2) node[anchor=north west,color=lgrey] {\huge $Q_y$};
\draw (2.5,7.7) node[anchor=north west,color=lgrey] {\huge $Y'$};

\begin{scriptsize}
\draw [fill=black] (7,8.5) circle (3.5pt);
\draw [fill=lgrey] (11.5,9.5) circle (3.5pt);
\end{scriptsize}
\end{tikzpicture}
}%
\subcaption{}
            \label{fig:NestedMaxSepsAreClRel1}
        \end{subfigure}
        \hspace{20mm}
         \begin{subfigure}[b]{0.35\linewidth}
            \definecolor{lgrey}{rgb}{0.4,0.4,0.4}
\definecolor{ligrey}{rgb}{0.8,0.8,0.8}
\definecolor{lgreen}{rgb}{0,0.8,0}
\scalebox{0.6}{%
\begin{tikzpicture}
\draw [line width=1.5pt,color=lgrey] (5,10)-- (4,8);
\draw [line width=1.5pt,color=lgrey] (4,7)-- (4.7,4.5);
\draw [line width=1.5pt,color=lgrey] (9,10)-- (13,8);
\draw [line width=1.5pt,color=lgrey] (13,7)-- (12,4.5);

\draw [line width=1.5pt, color=blue] (4.6,4)-- (12,4);


\draw [purple, line width=1.5pt] (4.59,5.5) -- (4.17,7.07);
\draw [purple, line width=1.5pt] plot [smooth, tension=0.3] coordinates {(4.59,5.5) (5.2,5.56) (6.2,4.2)};
\draw [purple,line width=1.5pt] (6.2,4.2) -- (8.75, 4.2);

\draw [purple, line width=1.5pt] plot [smooth, tension=0.3] coordinates {(8.75, 4.2)  (10,4.9) (12.13,5.25)};
\draw [purple, line width=1.5] (12.13,5.25) -- (12.83,7.05);

\draw [->, line width=1.5pt,color=purple] (10.5,4.97) -- (10.3,5.95);
\draw (10.6,5.8) node[purple,anchor=north west] {\LARGE $\vz$};

\draw [lgreen, line width=1.5pt] plot [smooth, tension=0.3] coordinates {(2.8,4.8) (5.1,5.36) (7.06,2.75) (10,4.7) (14,5.2)};

\draw [->,line width=1.5pt,color=blue] (7,4) -- (7,5.4);
\draw (7.3,5) node[anchor=north west,color=blue] {\LARGE $\vr$};

\draw [->,line width=1.5pt,color=lgreen] (9.6,4.5) -- (8.6,5.71);
\draw (9,5.8) node[anchor=north west,color=lgreen] {\LARGE $\vx'$};

\draw [->,line width=1.5pt,color=lgrey] (4.47,8.93) -- (5.17,8.45);
\draw [->,line width=1.5pt,color=lgrey] (4.17,6.38) -- (5.05,6.79);
\draw [->,line width=1.5pt,color=lgrey] (10.75,9.12) -- (10.36,8.3);
\draw [->,line width=1.5pt,color=lgrey] (12.88,6.7) -- (12.06,7);


\draw (7.2,9.5) node[anchor=north west] {\huge $P$};
\draw (13.5,4.3) node[orange,anchor=north west] {\huge $\cQ$};
\draw (2.5,7.7) node[anchor=north west,color=lgrey] {\huge $Y'$};

\begin{scriptsize}
\draw [fill=black] (7,8.5) circle (3.5pt);
\draw [fill=orange] (13,4.5) circle (3.5pt);
\draw [fill=orange] (4,4.6) circle (3.5pt);
\end{scriptsize}
\end{tikzpicture}
}%
\subcaption{}
            \label{fig:NestedMaxSepsAreClRel2}
        \end{subfigure}
        \centering
        \caption{Sketch for the proof of \cref{lem:NestedMaxSepsAreClRel}.}
        \label{fig:NestedMaxSepsAreClRel}
    \end{figure}
	
	We claim that $x'$ is nested with $Y'$ and thus the desired separation for the above argument. 
	To this end, suppose for a contradiction that $x'$ is not nested with $Y'$. Set
	\[
	\cQ := \{Q_y : \vy \in Y' \text{ and } x' \text{ crosses } y\},
	\]
	and first assume that there is a profile $Q_y \in \cQ$ with $\vx' \in Q_y$. Then it follows from $\yv$ being closely related to $Q_y$ that~$\yv \wedge \vx' \in \vS$. But this contradicts the choice of $\vx'$: as $y$ and $x'$ cross, $\yv \wedge \vx'$ is strictly smaller than $\vx'$. Moreover, since $\vr < \yv$ by assumption, we find that
	\[
	\vr \wedge (\yv \wedge \vx') = (\vr \wedge \yv) \wedge \vx' = \vr \wedge \vx' = \vr \wedge \vx. 
	\]
	Therefore, we may assume that $\xv' \in Q_y$ for all profiles $Q_y \in \cQ$.
	
	If $\vr \wedge \vx' \in \vS$, then we are done, so we may assume that $\vr \wedge \vx' \notin \vS$, which by the submodularity of $S$ implies that $\vr \vee \vx' \in \vS$. We then have~$(\vr \vee \vx')^* = \rv \wedge \xv' \leq \xv' \in Q_y$ for every $Q_y \in \cQ$ by the consistency of~$Q_y$. 
    Since each $\yv$ is closely related to~$Q_y$ by assumption, \cref{prop:InfimaOfClRelSetsAreClRel} implies that
	\[
	\vz := (\vr \vee \vx') \vee \bigvee_{Q_y \in \cQ} \vy = \Big((\vr \vee \vx')^* \wedge \bigwedge_{Q_y \in \cQ} \yv\Big)^* \in \vS
	\]
	(see \cref{fig:NestedMaxSepsAreClRel2}). By construction, $z$ is nested with $Y'$ and thus $z$ is nested with every separation in $Y$ that has an orientation which is greater than $\vr$. But since $\vr \leq \vz$, we have that $z$ is also nested with every separation in $Y$ that has an orientation which is smaller than $\vr$, and hence $z$ is nested with $Y$.
    Moreover, $\vz \in P$. Indeed, every $\vy \in Y$ is contained in $P$ by assumption, which implies that $\vz \in P$ as $P$ is a profile.
    Since $\vr \leq \vz$ by the definition of $z$, and $\vx' \leq \vz$ but~$\vx' \not\leq \vr$, it follows that $\vr < \vz$, which contradicts the maximality of $\vr$ in~$P_Y$. 
\end{proof}

We are now ready to prove that every $\vs_P$ is closely related to $P$.

\begin{LEM}\label{lem:spiscr}
	Every $\vs_P \in \vN(\vS, \cP)$ is closely related to $P$.
\end{LEM}

\begin{proof}
    We show by induction on $i$ that $\vs_P$ is closely related to $P$ for every $P \in \cP_i$. 
    Since $\vs_P = \inf(M_P)$ for every $P \in \cP_1$, the base case follows directly from \cref{prop:MaxSepsAreCloselyRelated} and \cref{prop:InfimaOfClRelSetsAreClRel}.
	So let $i > 1$, and suppose that every $\vs_{P'} \in \vN_{i-1}$ is closely related to $P'$. Since every $\vs_{P'} \in \vN_{i-1}$ is $P'$-exclusive for~$\cP_{\geq {i-1}}$, and thus $\sv_{P'} \in P$ for every such $\vs_{P'}$, we can apply \cref{lem:NestedMaxSepsAreClRel} to $P$ and $N_{i-1}$. This yields that $M_{P,i}$ is closely related to $P$, which by \cref{prop:InfimaOfClRelSetsAreClRel} implies that $\vs_P$ is closely related to $P$ as well.
\end{proof}

We are now left to show that for every inessential node $\sigma$ of $N(\vS, \cP)$, and every separation $\vs \in \sigma$, there is a profile $P \in \cP$ such that $\sv = \vs_P$.
Note that, by the construction of~$N(\vS, \cP)$, it is clear that $\sv \in \{\vs_P, \sv_P\}$ for some profile $P \in \cP$. What we need to show is that the orientation is the correct one, i.e.\ $\sv = \vs_P$. To this end, we first show that every profile $P$ lives at the unique node which contains $\vs_P$ if $s_P \in N$:

\begin{LEM}\label{lem:Sepsarecr}
	Given a profile $P$ such that $s_P$ exists, then $P$ lives at the unique node of $N(\vS, \cP)$ that contains $\vs_P$.
\end{LEM}

\begin{proof}
    Suppose for a contradiction that $P$ lives at a node $\sigma'$ of $N := N(\vS, \cP)$ which does not contain $\vs_P$. By \linebreak the definition of `node' and because $P$ lives at $\sigma'$, this implies that $\sigma'$ is the set of maximal elements in~$P \cap \vN$.\linebreak As $\vs_P$ is an element of $P \cap \vN$ but not of $\sigma'$, it follows that there is a separation $\vr \in \sigma'$ such that~$\vs_P < \vr$.
    
    By the definition of $N$, the separation $\vr$ has to be of the form $\vs_{P'}$ or $\sv_{P'}$ for some profile~$P' \in \cP$; however by \cref{lem:SepsInNDontPointTowardsEachOther} and because $\vs_P < \vr$ only $\vr = \vs_{P'}$ is possible.
    Therefore,~$\vs_P, \vs_{P'} \in P$. Further, by the definition of $\vs_{P'}$, we have $\vs_{P'} \in P'$, which, by the consistency of $P'$, implies that also~$\vs_P \in P'$. But then~$\vs_P, \vs_{P'} \in P \cap P'$, a contradiction to the exclusivity of either $\vs_P$ or $\vs_{P'}$, depending on which separation was added to $N$ first.
\end{proof}

It is now a simple corollary that the inessential nodes of $N(\vS, \cP)$ satisfy the conditions of \cref{lem:reflem}:

\begin{COR}\label{cor:Tciscool}
	For every inessential node $\sigma$ of $N(\vS,\cP)$ and every separation $\vs \in \sigma$ it holds that the inverse $\sv$ of $\vs$ is closely related to some profile in $\cP$.
\end{COR}

\begin{proof}
By the construction of $N(\vS, \cP)$, there is some profile $P \in \cP$ such that $\vs \in \{\vs_P, \sv_P\}$.
Since $\sigma$ is inessential, \cref{lem:Sepsarecr} implies that $\vs = \sv_P$. 
It follows from \cref{lem:spiscr} that $\sv$ is closely related to $P$.
\end{proof}

We are now ready to prove our second main result by applying \cref{lem:reflem} to $N(\vS, \cP)$:

\begin{THM}\label{thm:RefiningACanonicalToT}
	Let $\vS$ be a submodular separation system, and
	let $\cF$ be a friendly set of stars in $\vS$.
	Then there are nested sets $\tilde{N} \subseteq N \subseteq S$ such that:
	\begin{itemize}
		\item $\tilde{N}$ is canonical and distinguishes all the $\cF$-tangles of $S$;
		\item every node of $N$ is either a star in $\cF$ or home to an $\cF$-tangle.
	\end{itemize}
\end{THM}

\begin{proof}
	Let $\cP$ be the set of all $\cF$-tangles of $S$, and let $\tilde{N} := N(\vS, \cP)$ be the nested set from \cref{constr:Tc}.
	Further, let $\sigma =: \{\vs_1, ..., \vs_n\}$ be an inessential node of $\tilde{N}$. By \cref{cor:Tciscool} every separation $\sv_i$ is closely related to some profile $P_i \in \cP$. Thus, we can apply \cref{lem:reflem} to $\sigma$ to obtain a nested set~$N_\sigma \subseteq S$ corresponding to an $S$-tree over $\cF \cup \{\{\sv_1\}, ...,\{\sv_n\}\}$ in which each $\vs_i$ appears as a leaf separation. Clearly,~$N_\sigma$ is nested with $N$ and any other such set $N_{\sigma'}$, and thus setting\footnote{In fact, one can show that $\tilde{N}$ has at most one inessential node, and thus $N = \tilde{N} \cup N_\sigma$.} \linebreak $N := \tilde{N} \cup \bigcup \{N_\sigma : \sigma \text{ is an inessential node of } \tilde{N}\}$ yields the desired nested set.
\end{proof}

In particular, if the set $\cF$ of stars is invariant under all automorphisms of $\vS$, then \cref{thm:RefiningACanonicalToT} immediately implies \cref{cor:RefiningACanonicalToT}, which we restate here:

\begin{customthm}{\ref{cor:RefiningACanonicalToT}}
    \emph{Let $\vS$ be a submodular separation system, and let $\cF$ be a friendly set of stars in $\vS$. If $\cF$ is fixed under all automorphisms of $\vS$, then there are nested sets $\tilde{N} \subseteq N \subseteq S$ such that:
	\begin{itemize}
		\item $\tilde{N}$ is fixed under all automorphisms of $\vS$ and distinguishes all the $\cF$-tangles of $S$;
		\item every node of $N$ is either a star in $\cF$ or home to an $\cF$-tangle.
	\end{itemize}} 
\end{customthm}

\begin{proof}
    If $\cF$ is fixed under all automorphisms of $\vS$, then this is also true for the set $\cP$ of all $\cF$-tangles of $S$, i.e.\ $\phi(\cP) = \cP$ for all automorphisms $\phi$ of $\vS$. Applying \cref{thm:RefiningACanonicalToT} to $\vS$ and $\cF$ then yields nested sets $\tilde{N} := N(\vS, \cP) \subseteq N$ with the property that $\phi(N(\vS, \cP)) = N(\vS, \phi(\cP)) = N(\vS, \cP)$ for all such automorphisms $\phi$, and hence $\tilde{N}$ is fixed under them. 
\end{proof}

We conclude this section by remarking that it is not possible to strengthen \cref{thm:RefiningACanonicalToT} and \cref{cor:RefiningACanonicalToT} so that the refinement $N$, too, is canonical (see \cite{JoshRefining} for an example).

\section{A tree of tangles with good separations} \label{sec:ANearlyCanonicalToTWithGoodSeps}

Recall that a separation $s$ in a separation system $S$ is good (for a set $\cP$ of profiles of $S$) if there are two profiles $P \neq P'$ in $\cP$ such that $\vs$ is closely related to $P$ and $\sv$ is closely related to $P'$.
Additionally, we say that a separation $s \in S$ distinguishes two profiles $P$ and $P'$ of $S$ \emph{well} if one of $\vs$ and $\sv$ is closely related to~$P$ and the other one to~$P'$. Note that this implies that $s$ is good for $\cP$ if $P, P' \in \cP$.

 A set of separations is \emph{good for} some set $\cP$ of profiles, if every separation in that set is good for $\cP$. If we refer to a good set without reference to any set of profiles, then it can be assumed that this set is good for the set of all profiles of $S$.

In \cref{sec:RefiningInessStars} we proposed the property of a separation to be `good' as a structural generalization of efficiency to separation systems without an order function. 
Our main goal there was to find a property of separations which makes it possible to refine nested sets in the sense of \cref{lem:reflem}, and which can be formulated purely in terms of $\vS$ and $\cF$. While it is clear by \cref{prop:eff} that every separation which distinguishes some pair of profiles efficiently also distinguishes that pair well, it turned out that the converse is also true, at least for separation systems which come from a graph. More precisely, if~$G$ is a graph and~$S_k$ the system of all separations of~$G$ of order~$<k$, then it holds for every pair of regular profiles $P, P'$ of~$S_k$ that a separation in~$S_k$ distinguishes $P$ and $P'$ well if and only if it distinguishes them efficiently. A proof can be found in the extended version of this paper \cite{SARefiningInessPartsExtended}. 

Hence, the property of a separation to be good is not just some arbitrary generalization of efficiently distinguishing two profiles to separation systems without an order function, it in fact coincides in graphs with the notion of efficiency. Therefore, it seems to be a particularly natural generalization.
Since every separation system of the form $S_k$ contains a nested set which distinguishes all its profiles efficiently \cite{ProfilesNew}*{Theorem 3.6}, it seems natural to ask whether abstract separation systems always contain a nested set which distinguishes all its profiles well.

While this is not true in general \cite{SARefiningInessPartsExtended}, we will show a slightly weaker statement: that every submodular separation system contains a nested set of good separations which distinguishes all its profiles.
This nested set will be canonical but in a slightly weaker sense than usual. 
We show that this is best possible in the following sense: in \cref{ex:Counterexample_CanonicalGoodToT} we give an example of a submodular separation system and a set of profiles which does not admit a nested set of good separations that is canonical in the usual (stronger) sense.
\medskip

For the remainder of this section, let $\vS$ be a submodular separation system, and let $\cP$ be a set of profiles of $S$.
Our strategy for constructing a nested set that distinguishes $\cP$ and contains only separations that are good for $\cP$ will be as follows. We will first show that every pair of profiles is well distinguished by some separation in $S$. This then implies that there are enough good separations so that we can consider, similar as to \cref{constr:Tc}, the exclusive good separations in $\vS$ and use them to build a nested set that distinguishes already a subset of $\cP$, which we can then discard from $\cP$. In the next step, we then only consider the remaining profiles and those separations in $S$ which are nested with the separations we have chosen so far. Iterating this process will then yield the desired nested set.

Here, the main difference to \cref{constr:Tc} is that we only consider those exclusive separations that are maximal among all separations which are good for $\cP$. As these separations behave quite differently than the exclusive separations which are maximal among all separations in $\vS$, the proof of \cref{thm:ACanonicalGoodToT} will need different and additional arguments than the one showing that \cref{constr:Tc} works even though the constructions themselves are quite similar.
\medskip
 
During the construction, we need to make sure that every separation which we choose to include in our nested set $N$ of good separations is nested with all previously chosen separations. 
For this, throughout this section, let $M \subseteq S$ be a nested set of good separations such that $P \cap \vM = P' \cap \vM$ for all $P, P' \in \cP$. We set $S_M := \{s \in S : s \text{ is nested with } M\}$, and for every profile $P$ of $S$ we let $P_M := P \cap \vS_M$.

We first show that every pair of profiles in $\cP$ can be well distinguished by a separation which is nested with $M$. For this, we need the following fact about submodular separation systems.
A separation system $\vS$ inside some universe is called \emph{separable} if, for every pair of separations $\vs \leq \vr \in \vS$, there exists a separation $\vt \in \vS$ with $\vs \leq \vt \leq \vr$ such that $\vt$ emulates $\vs$ and $\tv$ emulates $\rv$.

\begin{LEM}{\cite{AbstractTangles}*{Lemma 13}} \label{lem:SubmodularImpliesSeparabel}
    Every submodular separation system is separable.
\end{LEM}

We can now use this fact to find good separations `between' two nested separations which distinguish some pair of profiles and are each closely related to one profile in that pair.

\begin{PROP}\label{prop:GoodSepsBetweenCloselyRelatedSeps}
    Let two profiles $P \neq P' \in \cP$ be given, and let $\vs \leq \vs' \in \vS$ be two separations that distinguish $P$ and $P'$. If $\sv$ is closely related to $P$ and $\vs'$ is closely related to $P'$, then there exists a separation $\vr \in \vS$ with $\vs \leq \vr \leq \vs'$ that distinguishes $P$ and $P'$ well.
    
    Moreover, if $s, s'$ are nested with $M$, then~$r$ can be chosen to be nested with $M$ as well.
\end{PROP}

\begin{proof}
    By \cref{lem:SubmodularImpliesSeparabel} there exists a separation $\vr \in \vS$ with $\vs \leq \vr \leq \vs'$ such that $\vr$ emulates $\vs$ and $\rv$ emulates $\sv'$. Since $\rv$ emulates $\sv'$, we have $\rv \vee \tv \in \vS$ and hence $(\rv \vee \tv)^* = \vr \wedge \vt \in \vS$ for all $\vt \leq \vs'$. 
    By \cref{prop:ShowingThatASepIsCloselyRelated} it follows that $\vr$ is closely related to $P'$ since $\vs'$ is closely related to $P'$. Analogously, we find that $\rv$ is closely related to $P$. 

    For the `moreover'-part let $\vt \leq \vt' \in \vS_M$ be two separations that minimize $\abs{\{x \in S : \vt \leq \vx \leq \vt'\}}$ under all separations in $S_M$ which distinguish $P$ and $P'$ and which have the property that $\tv$ is closely related to~$P$ and $\vt'$ is closely related to $P'$.
    For this note that $\vs, \vs'$ are candidates for $\vt, \vt'$. 

    By the first part, there exists a separation $r \in S$ with $\vt \leq \vr \leq \vt'$ which distinguishes $P$ and $P'$ well. If $\vt = \vr$, then $\vr$ is as desired since $t$ is nested with $M$, so we may assume that $\vt < \vr$.  Let $\vr' \in \vS$ be any maximal separation whose inverse $\rv'$ is closely related to $P$ and which satisfies $\vt \leq \vr' \leq \vt'$. By the argument before, we have $\vt < \vr'$.
    Then $r'$ cannot be nested with $M$; otherwise $\vr', \vt'$ would have been a better choice for $\vt, \vt'$. Moreover, $r'$ distinguishes $P$ and $P'$ well. Indeed, by the first part, there is a separation $r'' \in S$ with $\vr' \leq \vr'' \leq \vt'$ that distinguishes $P$ and $P'$ well, which by the maximal choice of $\vr'$ implies that $\vr' = \vr''$.

    Let $M' \subseteq M$ be the set of all separations in $M$ that cross $r'$, and denote the orientation of every $m \in M'$ which is contained in $P$ with $\vm$. Since every separation in $M$ distinguishes two profiles of $S$ well, there exists for every $m \in M'$ some profile~$Q_m$ of $S$ such that $\mv$ is closely related to~$Q_m$. Now first suppose that there is a separation $m \in M'$ such that $\rv' \in Q_m$. Then $\rv' \wedge \mv$ is an element of $\vS$ and closely related to~$P$ by \cref{prop:InfimaOfClRelSetsAreClRel}. But $\vr' < (\rv' \wedge \mv)^* \leq \vt'$, which contradicts the choice of $\vr'$.
    Therefore, $\vr' \in Q_m$ for every $m \in M$.
    Then $\vr'' := \vr' \wedge \bigwedge_{m \in M} \mv$ is closely related to~$P'$ by \cref{prop:InfimaOfClRelSetsAreClRel} as $\vr'$ is closely related to $P'$. But since $\vr'' < \vr' \leq \vt'$, this contradicts the choice of $\vt, \vt'$ as $\vt, \vr''$ would have been a better choice and thus concludes the proof.
\end{proof}

In particular, \cref{prop:GoodSepsBetweenCloselyRelatedSeps} implies that every pair of distinct profiles can be well distinguished:

\begin{COR}\label{cor:AllProfilesAreWellDistinguished}
    For every pair of distinct profiles in $\cP$, there is a separation in $S_M$ that distinguishes them well. 
\end{COR}

\begin{proof}
Let $P, P' \in \cP$ be two distinct profiles. We first show that there is a separation in $S_M$ which distinguishes them (but not necessarily well). Since $P$ and $P'$ are distinct, there exists a separation $\vs \in P$ that distinguishes them; we take one which is nested with as many separations in $M$ as possible. Now suppose for a contradiction that $s$ crosses some separation $m \in M$, and let $\vm \in P$. By the assumption on~$M$, there is a profile $Q$ of $S$ to which $\mv$ is closely related; by symmetry we may assume that $\vs \in Q$. Since $\mv$ is closely related to $Q$, it follows that $\mv \wedge \vs \in \vS$. But since $P \cap \vM = P' \cap \vM$ and therefore $\vm \in P'$, it follows that $\mv \wedge \vs = \vm \vee \sv \in P'$ as $P'$ is a profile. Hence, $\vm \wedge \vs$ distinguishes $P$ and $P'$, too, which contradicts the choice of $s$ since $\mv \wedge \vs$ is nested with one separation in $M$ more than $s$ by \cref{lem:Fischlemma}.

So let $\vs \in P_M$ be a maximal separation that distinguishes $P$ and $P'$, and let $\vs' \geq \sv$ be maximal in $P'_M$.
Then $\vs$ and $\vs'$ are closely related to $P$ and $P'$, respectively, by \cref{lem:NestedMaxSepsAreClRel}. Hence, by \cref{prop:GoodSepsBetweenCloselyRelatedSeps}, there exists a separation $r \in S_M$ with $\sv \leq \vr \leq \vs'$ that distinguishes $P$ and $P'$ well.
\end{proof}

Before we start with the construction of our desired nested set, we first show the following technical lemma which we will need throughout this section to find good separations:

\begin{LEM}\label{lem:GreaterGoodSeparations}
    Let $P, Q, Q' \in \cP$ be three distinct profiles, and let $\vr, \vs \in P_M$ be two distinct separations such that $\rv, \sv \in Q'$. If $\rv$ is closely related to $Q$ and $\sv$ is closely related to $Q'$, then there is a separation $\vu \in P_M$ with $\vr \vee \vs \leq \vu$ that distinguishes $P$ and $Q$ well.
\end{LEM}

\begin{proof}
    Since $\sv$ is closely related to $Q'$ and $\rv \in Q'$, it holds that $\rv \wedge \sv \in \vS$. Moreover, $\rv \wedge \sv \in \vS_M$ by \cref{lem:Fischlemma}.
    Now suppose that $\rv \wedge \sv$ is closely related to $Q$. Then the assertion follows. Since $P$ is a profile, $(\rv \wedge \sv)^* = \vr \vee \vs \in P$. Now let $\vw \in P_M$ be maximal with $\vr \vee \vs \leq \vw$. Then $\vw$ is closely related to $P$ by \cref{lem:NestedMaxSepsAreClRel}. It follows, by \cref{prop:GoodSepsBetweenCloselyRelatedSeps}, that there exists a separation $\vu \in \vS_M$ with $\vr \vee \vs \leq \vu \leq \vw$ that distinguishes~$P$ and~$Q$ well.
    
    Therefore, it suffices to show that $\rv \wedge \sv$ is closely related to $Q$.
    For this, it is by \cref{prop:ShowingThatASepIsCloselyRelated} enough to check that $(\rv \wedge \sv) \wedge \vt \in \vS$ for every separation $\vt \leq \rv \in Q$ since $\rv$ is closely related to $Q$ by assumption. So let $\vt \leq \rv$ be given. It follows that $\vt \in Q'$ by the consistency of $Q'$ as $\rv \in Q'$. Then, $(\rv \wedge \sv) \wedge \vt = \sv \wedge (\rv \wedge \vt) = \sv \wedge \vt \in \vS$ because $\sv$ is closely related to $Q'$. This completes the proof.
\end{proof}

We can now start with the construction of the desired nested set.
For this, let $E_P$ be the set of all separations in $P_M$ that are exclusive for $\cP$ and distinguish $P$ well from some other profile in $\cP$. We will pick, for every profile $P \in \cP$ for which $E_P$ is non-empty, a maximal separation in $E_P$. The set $N'$ of all these separations will then have two properties: on the one hand, it is nested, and on the other hand, the system of all those separations which are nested with $N'$ is still rich enough to distinguish all remaining profiles in~$\cP$.
This then allows us to continue the construction inductively.

To this end, we first show that $E_P$ is not empty for all profiles in $\cP$.

\begin{LEM}\label{lem:SomeEPIsNonEmpty}
    If $S$ and $\cP$ are non-empty, then $E_P$ is non-empty for some profile $P \in \cP$.
\end{LEM}

\begin{proof}
    Let $\vs \in \vS_M$ be a good separation such that the number of profiles $P \in \cP$ with $\vs \in P$ is as small as possible. We claim that $\vs$ is exclusive for some profile $P \in \cP$. It then follows that $\vs \in E_P$ since $s$ was good for $\cP$ and thus $s$ distinguishes $P$ well from some other profile. 
    
    To show this claim, suppose for a contradiction that there are distinct profiles $P, Q \in \cP$ with $\vs \in P \cap Q$. By \cref{cor:AllProfilesAreWellDistinguished}, there exists a separation $\vr \in P_M$ that distinguishes $P$ and $Q$ well. Moreover, as $s$ is good for $\cP$, there is a profile $Q' \in \cP$ such that $\sv$ is closely related to $Q'$; by symmetry we may assume that~$\rv \in Q'$.
    Then \cref{lem:GreaterGoodSeparations} applied to $P, Q, Q', \vr$ and $\vs$ yields a good separation $\vu \in P_M$ with $\vs \vee \vr \leq \vu$.
    Since~$\vs \leq \vu$, every profile which contains $\vu$ also contains $\vs$, but $\vs \in Q$ and $\vu \notin Q$ as $\uv \leq \rv \in Q$. So $\vu$ is contained in at least on profile less than $\vs$, which contradicts the choice of $\vs$.
\end{proof}

Next we show that every non-empty $E_P$ contains a unique maximal element. This will later allow us to prove that the nested set which we construct here is canonical in a certain sense.

\begin{LEM}\label{lem:UniqueMAxElementInEP}
    Every non-empty set $E_P$ contains a unique maximal element.
\end{LEM}

\begin{proof}
Suppose for a contradiction that there are at least two distinct maximal elements $\vr, \vs \in E_P$. Let~$Q$ and $Q'$ be the profiles that are well distinguished from~$P$ by $r$ and $s$, respectively. Since $\vr, \vs$ are exclusive for $P$, we have $\rv, \sv \in Q \cap Q'$. 
Applying \cref{lem:GreaterGoodSeparations} to $P, Q, Q', \vr, \vs$ yields a separation $\vu \in P_M$ with $\vr \vee \vs \leq \vu$ that distinguishes $P$ and $Q$ well. It follows that $\vu \in E_P$. 
Since both $\vr, \vs$ are maximal in $P_M$, they cross, and thus $\vu \geq \vr \vee \vs > \vr, \vs$, which contradicts the assumption that $\vr, \vs$ are maximal elements in~$E_P$.
\end{proof}

If $E_P$ is non-empty, we denote the unique maximal element of $E_P$ with $\vr_P$. The next step will be to show that picking the maximal separation of every non-empty $E_P$ actually yields a nested set of separations:

\begin{LEM}\label{lem:MaxGoodExclusiveSepsAreNested}
    Let $P$ and $P'$ be two distinct profiles with $E_P \neq \emptyset \neq E_{P'}$. Then $r_P$ and $r_{P'}$ are nested. 
\end{LEM}

\begin{proof}
Suppose for a contradiction that $r_P$ and $r_{P'}$ cross.
Since $\vr_{P'}$ is $P'$-exclusive and $\vr_P$ is $P$-exclusive, we have that $\rv_{P'} \in P$ and $\rv_P \in P'$. Let $P''$ be a profile from which $r_{P'}$ distinguishes $P'$ well. Then~$\rv_{P'}$ is closely related to $P''$. Applying \cref{lem:GreaterGoodSeparations} to $P', P, P'', \rv_{P}$ and $\vr_{P'}$ yields a separation $\vu \in P'_M$ with~$\rv_{P} \vee \vr_{P'} \leq \vu$ that distinguishes $P'$ and $P''$ well. But since $r_P$ and $r_{P'}$ cross, $\vr_{P'} < \vr_{P'} \vee \rv_{P} \leq \vu$, a contradiction to the maximality of $\vr_{P'}$ in $E_{P'}$.
\end{proof}

We are now ready to construct a nested set that distinguishes $\cP$ and uses only separations which are good for $\cP$, by applying the above lemmas. 

\begin{THM}\label{thm:ACanonicalGoodToT}
    Let $\vS$ be a submodular separation system, and let $\cP$ be a set of profiles of $S$. Then there exists a nested set $N := N(\vS, \cP)$ that distinguishes all profiles in $\cP$ and contains only separations which are good for $\cP$.
    
    Moreover, this set $N$ can be chosen so that if $\vS'$ is a submodular separation system and $\phi: \vS \rightarrow \vS'$ is an isomorphism of separation systems such that for all $\vr, \vs \in \vS$ we have that $\vr \vee \vs \in \vS$ if and only if $\phi(\vr) \vee \phi(\vs) \in \vS'$, then $\phi(N(\vS, \cP)) = N(\vS', \phi(\cP))$.
\end{THM}

\noindent We remark that this nested set will also be needed in a subsequent paper \cite{SARefiningEssParts}.
\medskip

Before we prove \cref{thm:ACanonicalGoodToT}, let us make a few more remarks about its canonicity statement.
The definition of canonicity in \cref{thm:ACanonicalGoodToT} is weaker than the usual one, which requires $\phi(N(\vS, \cP)) = N(\vS', \phi(\cP))$ for all isomorphisms of separation systems $\phi: \vS \rightarrow \vS'$ and not just for those which preserve infima and suprema, as we do here.
We need this stronger assumption on the isomorphisms to make sure that the set of all good separations is invariant under every such $\phi$, i.e.\ if $s \in S$ is a good separation for $\cP$ and $\phi: \vS \rightarrow \vS'$ is as above, then $\phi(s)$ is a good separation for $\phi(\cP)$ and vice versa. Indeed, as the existence of infima and suprema is preserved under every $\phi$ which satisfies the stronger assumptions, this in particular implies that a separation $\vs \in P$ is closely related to $P$ if and only if $\phi(\vs)$ is closely related to $\phi(P)$.

\begin{proof}[Proof of \cref{thm:ACanonicalGoodToT}]
    We show by induction on $\abs{\cP}$ that for every nested set $M$ of good separations which satisfies that $P \cap \vM = P' \cap \vM$ for all $P, P' \in \cP$ there exists a nested set $N(\vS, \cP) \subseteq S_M$ which distinguishes~$\cP$ and is good for $\cP$. Moreover, this set will be canonical in the above sense if $M$ is canonical in that sense. With $M := \emptyset$ this clearly implies the assertion.

    So let some such set $M$ be given.
    If $\abs{\cP} \leq 1$, then $N := \emptyset$ is a valid choice. So we may assume that $\abs{\cP} > 1$ and that the assertion holds for all sets $\cP'$ with $\abs{\cP'} < \abs{\cP}$. 
    
    Set $N' := \{r_P : P \in \cP \text{ and } E_P \neq \emptyset\}$. By \cref{lem:MaxGoodExclusiveSepsAreNested} the set $N'$ is nested. Moreover, $N'$ consists only of separations that are good for $\cP$ and is clearly an invariant of $\vS$, $\cP$ and $M$ under all isomorphisms that preserve infima and suprema.
    Let $\cP' \subseteq \cP$ be the set of all profiles $P \in \cP$ with empty $E_P$, and let~$M' := M \cup N'$. Note that, by construction, $N'$ distinguishes all profiles in $\cP \setminus \cP'$ from all other profiles in~$\cP$. Moreover, $\rv_P \in P'$ for all $P \in \cP \setminus \cP'$ and $P' \in \cP'$, which implies that $P \cap \vM' = P' \cap \vM'$ for all~$P' \in \cP'$.
    
    Since $\abs{\cP'} < \abs{\cP}$ by \cref{lem:SomeEPIsNonEmpty}, we can apply the induction hypothesis to $\cP'$ and $M'$ to obtain a nested set $N''$ that distinguishes $\cP'$ well and is an invariant of $\vS$, $\cP'$ and $M'$ in the above sense. Since $M'$ and~$\cP'$ themselves are invariants of $M$ and $\cP$ in the above sense, we have that $N' \cup N''$ is the desired set.
\end{proof}

We conclude this section with showing that one cannot strengthen \cref{thm:ACanonicalGoodToT} with regards to the canonicity statement. More precisely, the following example describes a submodular separation system and a set of profiles such that there is no nested set that distinguishes $\cP$, contains only good separations and is canonical in the usual (stronger) sense.

\begin{EX}\label{ex:Counterexample_CanonicalGoodToT}
    Let $V$ consists of the 20 grey points in \cref{fig:CounterexampleCanonicalGoodToT}, and $\vU$ be the universe of all bipartitions of $V$, i.e.\ $\vU$ contains $(A,B)$ and $(B,A)$ for each bipartition of $V$ (compare \cite{TangleTreeGraphsMatroids}). Further, let $S$ be the separation system given by the 20 bipartitions of $V$ outlined in \cref{fig:CounterexampleCanonicalGoodToT} together with $\{V, \emptyset\}$ and all their corner separations which do not distinguish $P_1$ and $P_2$. Here, $P_1$ and $P_2$ are the profiles one obtains by orienting every separation in $S$ to the side that contains $v_1$ or $v_2$, respectively.
    It is straight forward to check that $\vS$ is submodular and that $\vs_1 \vee \vr_1$ and $\vs_2 \wedge \vr_2$ are not in $\vS$. 
    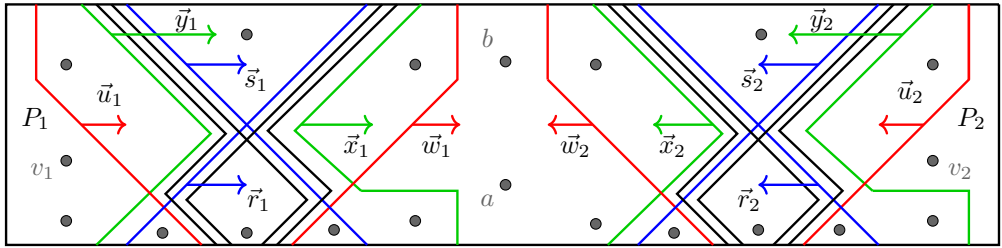
\begin{figure}[h!]
		\centering
        \definecolor{grey}{rgb}{0.4,0.4,0.4}
\definecolor{lgreen}{rgb}{0,0.8,0}
\scalebox{1}{%
	\begin{tikzpicture}[scale=4]
		\draw [line width=0.9pt] (4.1,1.6)-- (4.1,0.8)-- (0.8,0.8)-- (0.8,1.6)-- (4.1,1.6);
		
		\draw [line width=0.9pt,color=lgreen] (2.18,1.6)-- (1.76,1.18)-- (1.98,0.98);
		\draw [line width=0.9pt,color=lgreen] (1.98,0.98)-- (2.3,0.98)-- (2.3,0.8);
		\draw [line width=0.9pt,color=lgreen] (1.05,1.6)-- (1.48,1.17)-- (1.1,0.8);
		\draw [line width=0.9pt,color=lgreen] (3.88,1.6)-- (3.46,1.18)-- (3.68,0.98);
		\draw [line width=0.9pt,color=lgreen] (3.68,0.98)-- (4,0.98)-- (4,0.8);
		\draw [line width=0.9pt,color=lgreen] (2.75,1.6)-- (3.18,1.17)-- (2.8,0.8);
		
		\draw [line width=0.9pt,color=red] (1.75,0.8)--(2.3,1.35)-- (2.3,1.6);
		\draw [line width=0.9pt,color=red] (1.45,0.8)-- (0.9,1.35)-- (0.9,1.6);
		\draw [line width=0.9pt,color=red] (3.45,0.8)--(4,1.35)-- (4,1.6);
		\draw [line width=0.9pt,color=red] (3.15,0.8)-- (2.6,1.35);
		\draw [line width=0.9pt,color=red] (4,1.35)-- (4,1.6);
		\draw [line width=0.9pt,color=red] (2.6,1.35)-- (2.6,1.6);
		
		\draw [line width=0.9pt,color=blue] (1.2,1.6)-- (2,0.8);
		\draw [line width=0.9pt,color=blue] (1.2,0.8)-- (2,1.6);
		\draw [line width=0.9pt,color=blue] (2.9,1.6)-- (3.7,0.8);
		\draw [line width=0.9pt,color=blue] (2.9,0.8)-- (3.7,1.6);
		
		\draw [line width=0.9pt] (1.55,0.8)-- (1.4,0.95)-- (2.05,1.6);
		\draw [line width=0.9pt] (1.65,0.8)-- (1.8,0.95)-- (1.15,1.6);
		\draw [line width=0.9pt] (1.7,0.8)-- (1.88,0.98)-- (1.67,1.18)-- (2.1,1.6);
		\draw [line width=0.9pt] (1.5,0.8)-- (1.33,0.97)-- (1.53,1.17)-- (1.1,1.6);
		\draw [line width=0.9pt] (3.25,0.8)-- (3.1,0.95)-- (3.75,1.6);
		\draw [line width=0.9pt] (3.35,0.8)-- (3.5,0.95)-- (2.85,1.6);
		\draw [line width=0.9pt] (3.4,0.8)-- (3.58,0.98)-- (3.37,1.18)-- (3.8,1.6);
		\draw [line width=0.9pt] (3.2,0.8)-- (3.03,0.97)-- (3.23,1.17)-- (2.8,1.6);

		\draw [->,line width=0.9pt,color=blue] (1.4,1.4) -- (1.6,1.4);
		\draw [->,line width=0.9pt,color=blue] (1.4,1) -- (1.6,1);
		\draw [->,line width=0.9pt,color=blue] (3.5,1.4) -- (3.3,1.4);
		\draw [->,line width=0.9pt,color=blue] (3.5,1) -- (3.3,1);
		\draw [->,line width=0.9pt,color=red] (2.15,1.2) -- (2.31,1.2);
		\draw [->,line width=0.9pt,color=red] (1.05,1.2) -- (1.2,1.2);
		\draw [->,line width=0.9pt,color=red] (2.75,1.2) -- (2.6,1.2);
		\draw [->,line width=0.9pt,color=red] (3.85,1.2) -- (3.7,1.2);
		\draw [->,line width=0.9pt,color=lgreen] (1.78,1.2) -- (2.02,1.2);
		\draw [->,line width=0.9pt,color=lgreen] (3.15,1.2) -- (2.95,1.2);
		\draw [->,line width=0.9pt,color=lgreen] (1.15,1.5) -- (1.5,1.5);
		\draw [->,line width=0.9pt,color=lgreen] (3.78,1.5) -- (3.4,1.5);

		\draw (3.21,1.42) node[anchor=north west] {$\vec{s}_2$};
		\draw (3.2,1.02) node[anchor=north west] {$\vec{r}_2$};
		\draw (2.61,1.2) node[anchor=north west] {$\vec{w}_2$};
		\draw (3.73,1.38) node[anchor=north west] {$\vec{u}_2$};
		\draw (1.56,1.42) node[anchor=north west] {$\vec{s}_1$};
		\draw (1.57,1.02) node[anchor=north west] {$\vec{r}_1$};
		\draw (1.07,1.37) node[anchor=north west] {$\vec{u}_1$};
		\draw (2.15,1.2) node[anchor=north west] {$\vec{w}_1$};
		\draw (1.33,1.62) node[anchor=north west] {$\vec{y}_1$};
		\draw (1.89,1.2) node[anchor=north west] {$\vec{x}_1$};
		\draw (2.94,1.2) node[anchor=north west] {$\vec{x}_2$};
		\draw (3.44,1.62) node[anchor=north west] {$\vec{y}_2$};

		\draw (0.82,1.288) node[anchor=north west] {$P_1$};
		\draw (3.93,1.288) node[anchor=north west] {$P_2$};
		\draw [color=grey](0.85,1.1) node[anchor=north west] {$v_1$};
		\draw [color=grey](3.9,1.1) node[anchor=north west] {$v_2$};
		\draw [color=grey](2.4,1.0) node[anchor=north] {$a$};
		\draw [color=grey](2.4,1.55) node[anchor=north] {$b$};

		\begin{scriptsize}
			\draw [fill=grey] (1,1.08) circle (0.5pt);
			\draw [fill=grey] (1,1.4) circle (0.5pt);
			\draw [fill=grey] (1,0.88) circle (0.5pt);
			\draw [fill=grey] (1.6,1.5) circle (0.5pt);
			\draw [fill=grey] (1.6,0.84) circle (0.5pt);
			\draw [fill=grey] (1.32,0.84) circle (0.5pt);
			\draw [fill=grey] (1.89,0.85) circle (0.5pt);
			\draw [fill=grey] (2.16,0.88) circle (0.5pt);
			\draw [fill=grey] (2.16,1.4) circle (0.5pt);
			\draw [fill=grey] (2.46,1) circle (0.5pt);
			\draw [fill=grey] (2.76,1.4) circle (0.5pt);
			\draw [fill=grey] (3.02,0.84) circle (0.5pt);
			\draw [fill=grey] (3.3,0.85) circle (0.5pt);
			\draw [fill=grey] (3.31,1.5) circle (0.5pt);
			\draw [fill=grey] (3.88,1.4) circle (0.5pt);
			\draw [fill=grey] (3.88,0.88) circle (0.5pt);
			\draw [fill=grey] (3.88,1.08) circle (0.5pt);
			\draw [fill=grey] (3.57,0.85) circle (0.5pt);
			\draw [fill=grey] (2.46,1.41) circle (0.5pt);
			\draw [fill=grey] (2.76,0.87) circle (0.5pt);
		\end{scriptsize}
	\end{tikzpicture}
}%
		\caption{A separation system and two profiles without a canonical good tree of tangles.}
		\label{fig:CounterexampleCanonicalGoodToT}
	\end{figure}    
	
	Consider the map $\phi: \vS \rightarrow \vS$ that maps $\vx_1$ and $\vy_1$ to $\vx_2$ and $\vy_2$, respectively, and every other separation to its reflection that one obtains by mirroring on the vertical axis through $a$ and $b$.
	Then $\phi$ is an automorphism of separation systems since an automorphism only has to preserve the partial order on $\vS$, and here $\vx_1$ is the unique supremum of $\vs_1$ and $\vr_1$ in $\vS$ even though it is not their supremum $\vs_1 \vee \vr_1$ in $\vU$ (and analogously for $\vy_2$). 
	
	It is straight forward to check that $\vs_1 \wedge \vw_1$, $\vy_1 \vee \vu_1$, $\vr_2 \vee \vu_2$ and $\vx_2 \wedge \vw_2$ are the only separations in $\vS$ that are good for $\cP := \{P_1, P_2\}$. But $\phi$ maps none of these separations to another good separation. Hence, there is no nested set $N$ that distinguishes~$P_1$ and $P_2$, contains only separations that are good for $\cP$, and satisfies $\phi'(N) = N$ for all automorphisms $\phi'$ of $\vS$.
\end{EX}

\section{Outlook}\label{sec:Outlook}

The nested sets from \cref{cor:RefiningACanonicalToT} and \cref{thm:RefiningGoodToTs} have the property that the inessential nodes are too `small' to be home to a tangle in that they are stars in $\cF$. However, we did not require the essential nodes to be `small' in any sense. 
Indeed, if the separation system at hand comes from a graph $G$, then it is not too difficult to see that the nested set $N$ from \cref{cor:RefiningACanonicalToT} and \ref{thm:RefiningGoodToTs} directly translates to a tree-decomposition of $G$, and the essential parts of that decomposition might contain a great portion of $G$ with lots of vertices that do not really `belong' to the tangle living in that part. It would therefore be nice to know if it is possible to further refine the essential nodes of $N$ as well.

One natural approach would be to ask the essential nodes of $N$ to be `maximal' in the tangle they are home to, while preserving the property that every inessential node of $N$ is a star in $\cF$.
For this note that the partial order on $\vS$ induces a partial order on the set of all `proper' stars \cite{TreeSets}. 
A star $\sigma \subseteq \vS$ is \emph{proper} if for every distinct $\vs, \vr \in \sigma$ the relation $\vs \leq \rv$ is the only one, i.e.\ $\vs \not\leq \vr$, $\vs \not\geq \vr$ and~$\vs \not\geq \rv$. For two proper stars~$\sigma, \tau \subseteq \vS$ we have $\sigma \leq \tau$ if and only if for every $\vs \in \sigma$ there exists some~$\vr \in \tau$ such that $\vs \leq \vr$.
A proper star in $\vS$ is \emph{maximal} in a profile $P$ of $S$ if it is a maximal element in the set of all proper stars in~$P$.

While it is not true that the essential nodes of the nested set $N(\vS, \cP)$ from \cref{constr:Tc} can be refined in that way \cite{SARefiningInessPartsExtended}, we will show in \cite{SARefiningEssParts} that good tree sets indeed always admit such refinements. 
More precisely, if $S$ is a submodular separation system inside some distributive universe, then it is possible to refine every tangle-distinguishing tree set which is good for the set of all $\cF$-tangles so that all its inessential nodes are stars in $\cF$, and all its essential nodes are maximal in the tangle living at them \cite{SARefiningEssParts}.

Moreover, we will show a stronger statement for separations systems coming from graphs which will circumvent the following problem. 
Let $G$ be some graph and $k \in \N$, and let $\sigma$ be an essential node of a tree set $\tilde{N}$ inside $S_k$.
Then one can always refine $\sigma$ with the rather naive star which contains all small separations of the form $(A,V(G))$, where the small sides $A$ cover $\bigcap_{(C,D) \in \sigma} D$. The arising star $\sigma'$ will then be maximal in the tangle living at it, and all newly arising inessential nodes will be of the form $\{(V(G),A)\}$ and hence be co-small. 
However, in doing so, we have not gained any information about the graph $G$. Indeed, if we compare the two tree-decompositions corresponding to $\tilde{N}$ and its refinement described above, then the essential parts corresponding to $\sigma$ and $\sigma'$ will be the same.
Ideally, though, we would like to make the essential parts as small as possible, so we are more interested in essential stars whose \emph{interior} $\bigcap_{(A,B) \in \sigma'} B$ is as small as possible.
In \cite{SARefiningEssParts} we show that such refinements always exist\footnote{The problem of refining essential stars in graphs was already addressed by Erde \cite{JoshRefining}. \cref{thm:GraphTanglesToTNestedSetVersion} strengthens his result.}:

\begin{THM}\label{thm:GraphTanglesToTNestedSetVersion}
    Let $G$ be a graph, $k \in \N$, and let $\cF$ be a friendly set of stars in $\vS_k(G)$. 
    Further, let $\tilde{N} \subseteq S_k(G)$ be a nested set of separations that distinguishes the all the $\cF$-tangles of~$S_k(G)$ so that every separation in~$\tilde{N}$ efficiently distinguishes a pair of $\cF$-tangles of $S_k(G)$. Then there exists a nested set $N \subseteq S_k(G)$ with $\tilde{N} \subseteq N$ such that
    \begin{enumerate}[label=\rm{(\roman*)}]
        \item\label{itm:GraphTanglesToTNSiness} every inessential node of $N$ is a star in $\cF$;
        \item\label{itm:GraphTanglesToTNSess} the interior of every essential node of $N$ is of smallest size among all the exclusive stars contained in the $\cF$-tangle living at that node.
    \end{enumerate}
\end{THM}

\noindent Note that applying \cref{thm:GraphTanglesToTNestedSetVersion} to the tree set from \cite{CDHH13CanonicalAlg} yields that there exists a \emph{canonical} tree set which can be refined so that both its essential and inessential nodes are small.

\bibliographystyle{amsplain}
\bibliography{collective}

\end{document}